\documentclass[9pt,shortpaper,twoside,web]{ieeecolor}
\usepackage{generic}
\usepackage{textcomp}
\def\BibTeX{{\rm B\kern-.05em{\sc i\kern-.025em b}\kern-.08em
    T\kern-.1667em\lower.7ex\hbox{E}\kern-.125emX}}
\usepackage[noadjust]{cite}
\usepackage{amsmath,amssymb,amsfonts}
\usepackage{bigints}
\usepackage{abraces}
\usepackage{mathrsfs}
\usepackage{color}
\usepackage{graphicx}
\usepackage{epsfig} 
\usepackage{hyperref}
\usepackage{subcaption}
\usepackage{caption}

\usepackage{array,booktabs} 
\usepackage{amsmath,amssymb} 
\usepackage{stfloats} 
\usepackage{psfrag} 
\usepackage{color}

\definecolor{box_color}{rgb}{.8,.8,.8}

\usepackage[utf8]{inputenc}
\usepackage[T1]{fontenc}
\usepackage{epstopdf}

\newtheorem{lemma}{Lemma}
\newtheorem{proposition}{Proposition}
\newtheorem{corollary}{Corollary}
\newtheorem{fact}{Fact}

\newtheorem{remark}{Remark}
\newtheorem{example}{Example}

\newtheorem{assumption}{Assumption}
\newtheorem{ass}{C.}

\usepackage{tikz,xcolor,hyperref}

\definecolor{lime}{HTML}{A6CE39}
\DeclareRobustCommand{\orcidicon}{
	\begin{tikzpicture}
	\draw[lime, fill=lime] (0,0) 
	circle [radius=0.16] 
	node[white] {{\fontfamily{qag}\selectfont \tiny ID}};
	\draw[white, fill=white] (-0.0625,0.095) 
	circle [radius=0.007];
	\end{tikzpicture}
	\hspace{-2mm}
}

\foreach \x in {A, ..., Z}{\expandafter\xdef\csname orcid\x\endcsname{\noexpand\href{https://orcid.org/\csname orcidauthor\x\endcsname}
			{\noexpand\orcidicon}}
}

\input epsf

\def\begcen{\begin{center}}
\def\endcen{\end{center}}

\newcommand{\bfx}{\mbox{$x$}}

\def\et{\epsilon_t}
\def\calh{\mathcal{H}}
\def\calk{\mathcal{K}}
\def\ms{{\mathfrak s}}

\def\bfx{{\bf x}}

\newcommand{\RE}{\mathbb {R}}    

\newcommand{\blue}[1]{{\color{blue}{#1}}}

\newcommand{\col}{ \mbox{col} }

\def\calm{{\cal M}}
\def\call{{\cal L}}
\def\nea{\mathbb{N}}

\def\hal{{1 \over 2}}

\def\L2{{\cal L}_2}
\def\L2e{{\cal L}_{2e}}

\def\rea{\mathbb{R}}

\def\x{{x}}

\def\begequarr{\begin{eqnarray}}
\def\endequarr{\end{eqnarray}}
\def\begequarrs{\begin{eqnarray*}}
\def\endequarrs{\end{eqnarray*}}
\def\begarr{\begin{array}}
\def\endarr{\end{array}}
\def\begequ{\begin{equation}}
\def\endequ{\end{equation}}
\def\begmat#1{\begin{bmatrix}#1\end{bmatrix}}
\def\lab{\label}
\def\begdes{\begin{description}}
\def\enddes{\end{description}}
\def\begenu{\begin{enumerate}}
\def\begite{\begin{itemize}}
\def\endite{\end{itemize}}
\def\begmyite{\begin{myitemize}}
\def\endmyite{\end{myitemize}}
\def\endenu{\end{enumerate}}

\def\lef[{\left[\begin{array}}
\def\rig]{\end{array}\right]}
\def\qed{\hfill$\Box \Box \Box$}
\def\begcen{\begin{center}}
\def\endcen{\end{center}}
\def\begrem{\begin{remark}\rm}
\def\endrem{\end{remark}}
\def\begassum{\begin{assumption}}
\def\endassum{\end{assumption}}
\def\begassums{\begin{assumption*}}
\def\endassums{\end{assumption*}}
\def\begassu{\begin{ass}}
\def\endassu{\end{ass}}
\def\beglem{\begin{lemma}}
\def\endlem{\end{lemma}}
\def\begcor{\begin{corollary}}
\def\endcor{\end{corollary}}
\def\begfac{\begin{fact}}
\def\endfac{\end{fact}}
\def\TAC{{\it IEEE Trans. Autom. Control}}
\def\AUT{{\it Automatica}}

\def\SCL{{\it Syst. Control Lett.}}
\def\IJRNLC{{\it Int. J. Robust Nonlinear Control}}
\def\CST{{\it IEEE Trans. Control Syst. Technol.}}

\definecolor{gre}{rgb}{0.0, 0.5, 0.0}

\begin{document}

\title{Immersion and Invariance Orbital Stabilization of Underactuated Mechanical Systems with Collocated Pre-Feedback}
\author{Jose~Guadalupe~Romero\hspace{-0.5mm}\orcidB{}, and Bowen Yi\hspace{-0.5mm}\orcidA{}
\thanks{J.G. Romero is with Departamento Acad\'emico de Sistemas Digitales, ITAM, R\'io Hondo 1,  01080, Ciudad de M\'exico, M\'exico and also with the Department of Mechanical Engineering, The
Hong Kong Polytechnic University (PolyU), KLN, Hong Kong, e-mail: {\tt jose.romerovelazquez@itam.mx}}
\thanks{B. Yi is with Australian Centre for Field Robotics, The University of Sydney, NSW 2006, Australia, e-mail: {\tt b.yi@outlook.com} (corresponding author)}
}
\maketitle
\begin{abstract}
In this note we study the generation of attractive oscillations of a class of mechanical systems with underactuation one. The proposed design consists of two terms, {\em i.e.}, a partial linearizing state feedback, and an immersion and invariance orbital stabilization controller. The first step is adopted to simplify analysis and design, however, bringing an additional difficulty that the model loses its Euler-Lagrange structure after the collocated pre-feedback. To address this, we propose a constructive solution to the orbital stabilization problem via a smooth controller in an analytic form, and the model class identified in the paper is characterized via some easily \emph{apriori} verifiable assumptions on the inertia matrix and the potential energy function.
\end{abstract}
%
%
\IEEEpeerreviewmaketitle

\section{Introduction}
\lab{sec1}
%

Oscillating behaviour of dynamical systems is ubiquitous in biology, physics and engineering \cite{WIN,STR,SCHSAN}. For the latter, we are concerned with \emph{constructive} approaches to generate stable oscillators in closed loop---known as the orbital stabilization problem \cite{KHA}---which widely appears in many engineering areas, including bipedal robots \cite{SPOBUL}, exoskeletons \cite{LVetal}, electric motors \cite{MARbook}, AC power converters \cite{ORTbook}, path following \cite{NIEetal,YIpath}, combustion oscillations, etc. Despite the fact that these control problems may be formulated as trajectory tracking, it brings the following merits to consider them as orbital stabilization:
\begin{enumerate}
\item[1)] There is no need to ensure phase synchronization, {\em i.e.} asymptotic convergence of tangential coordinates (also known as angular variable or isochrons), making the control design procedure more flexible;
\item[2)] The closed-loop dynamics is autonomous, rather than being time-varying in tracking control. As a result, it is unnecessary to adopt motion planning algorithms to generate reference trajectories \cite{SAEetal,SHIetal};
\item[3)] The error dynamics in tracking control usually fails to preserve geometric properties of the original control system. For example, the tracking error dynamics of mechanical systems in general are not in Euler-Lagrange (EL) or Hamiltonian forms, unless imposing additional assumptions \cite{ROMetaltac}.
\end{enumerate}

Unlike point regulation, there are only a few tools for orbital stabilization, which may be generally classified into the following two categories. The first class is based on the concept of energy, see for examples \cite{FRAPOG,YIetal,ARAetal,SPOBUL,STASEP}, which rely on either shaping the total energy with minima on the desired orbit, or regulating the energy to some value. In \cite{YIetal}, these two ideas are systematically studied from the perspective of port-Hamiltonian systems, and their equivalence has been revealed. Another widely studied technical route builds upon a geometric perspective. In \cite{BANHAU}, the authors consider the case that in a neighborhood of a given orbit a change of coordinate may be found locally to decompose the systems state into transverse coordinates and phase variable, and together with a feedback transformation, the resulting transverse dynamics are linear in order to deal with orbital stabilization. Later, an idea based on transverse linearisation was elaborated for underactuated EL mechanical systems in \cite{SHIetal,SHIetal2010}, known as the virtual holonomic constraint (VHC) approach, which may conceptually date back to \cite{APP} published in 1911; see also its applications in robotics \cite{WESetal}. The key step is using feedback linearization to impose a geometric constraint of configuration variables with respect to a new independent variable, with the obtained zero dynamics behaving as oscillations. Note that a transverse linearization requires knowledge of a particular periodic orbit. Recently, the immersion and invariance (I\&I) technique, which was originally proposed for nonlinear and adaptive control and observers design \cite{IIbook}, is adopted for orbital stabilization of nonlinear systems \cite{Ortetal20}. This method can be viewed as an extension of VHC to general nonlinear systems, but with a key difference of the design procedure. Namely, in I\&I a lower dimensional-oscillator needs to be selected in advance, then solving the Francis-Byrnes-Isidori (FBI) equations.

In this paper, the I\&I orbital stabilization approach is tailored for a class of underactuated mechanical systems.\footnote{In this paper the terminology ``orbital stabilisation'' refers to generation of stable oscillators with a group of feasible orbits. Note that in some works it refers to the convergence to a particular orbit \cite{MOHetal2,SHIetal}. } To be precise, the main contributions are summarized as follows.

\begin{itemize}
\item[\bf C1] We identify a class of underactuated mechanical systems, which can be orbitally stabilized via integrating collocated feedback linearization with an elaborated I\&I controller, and obtain a compact {\it analytic} formula. 

%
%
\item[\bf C2] In the VHC approach, there are two key problems concerned, {\em i.e.}, stabilizing systems state to a constraint manifold and analyzing motion properties of the reduced dynamics on the manifold. However, full-state boundedness in closed loop is rarely discussed, which may be not guaranteed in many cases.\footnote{
The reduced dynamics in VHC is obtained on the constraint manifold. However, there is an asymptotically decaying term in addition to the (nominal) reduced dynamics when the state converges to the manifold in the transient stage, which may cause unboundedness of the closed loop.} To address this, the state boundedness is analyzed comprehensively in the paper. 

%
\end{itemize}

The remainder of the paper is organized as follows. The EL model which we are interested in and the precise problem formulation are described in Section \ref{sec3}. We present the main result in Section \ref{sec4}, and give more details on connection to other related works in Section \ref{sec:discussion}. In Section \ref{sec5} we use two benchmark models to demonstrate the controller design procedure and evaluate the performance of the proposed scheme. In Section \ref{sec6} we present some concluding remarks. Finally, to make the article as self-contained as possible, the description of the I\&I orbital stabilization technique \cite{Ortetal20} is recalled in Appendix.

\vspace{0.2cm}
%
\noindent {\bf Notation.} Throughout the paper, we assume all mappings $C^2$-continuous. $I_n$ is the $n \times n$ identity matrix and $0_{n \times s}$ is an $n \times s$ matrix of zeros. For $x \in \rea^n$, $S \in \rea^{n \times n}$, $S=S^\top >0$, we denote the Euclidean norm $|x|^2:=x^\top x$, and the weighted--norm $\|x\|^2_S:=x^\top S x$. Given a function $f: \rea^n \to \rea^m$ we define the differential operators
$
\nabla f:= ({\partial f \over \partial x})^\top,
~
\nabla_{x_i} f:= ({\partial f \over \partial x_i})^\top,
%
$
in which $x_i$ is the $i$-th element of the vector $x$. We use $\mathbb{S}$ to denote the unit circle, and $g^\perp:\rea^n \to \rea^{(n-m) \times n}$ as a full-rank left-annihilator of $g(x) \in \rea^{n\times m}$. For a multi-variable smooth function $V(x,\xi)$, when clear we use ${\partial V \over \partial \xi}(a,b)$ to denote $ {\partial V (x,\xi)\over \partial \xi}\big|_{x=a,\xi=b}$. 

\section{Mechanical systems and problem formulation}
\lab{sec3}
%
In this note we consider underactuated mechanical systems with two degrees of freedom (DOF's)\footnote{We adopt this assumption to simplify the presentation. Indeed, the proposed approach can be extended straightforwardly to underactuated mechanical systems with arbitrary DOF's and degree of underactuation one.}, the dynamics of which is described by the EL equations of motion
 \begequarr
 \label{lagr}
M(q) \ddot{q} + C(q,\dot{q}) \dot{q} + \nabla V(q) = G  \tau,
 \endequarr
in which $q \in  \mathbb{S} \times \rea$ are the configuration variables, $\tau \in \RE^1$ is the control input, $M(q) >0$ is the generalized inertia matrix, $C(q,\dot{q})\dot q$ represents the Coriolis and centrifugal forces, $V(q)$ is the potential energy. To simplify the notation, we define $q=\col(q_u,q_a)$, with $q_a$ as the actuated coordinate and $q_u$ the unactuated one and, without loss of generality, the  input matrix is of the form
$$
G= \left[ \begarr{cc} 0 & 1 \endarr \right]^\top.
$$

Now,  the class of systems to be addressed  in the note are characterized  by the the following assumption.

\begin{itemize}
\item[{{\bf A1.}}] All elements of inertia matrix depend only on the unactuated variable $q_u$ and it has the form
 \begin{equation}
 \label{matM}
 M (q_u)= \left[ \begarr{cc} m_{uu} (q_u) & m_{au} (q_u) \\ m_{au} (q_u) & m_{aa}  (q_u) \endarr \right].
 \end{equation}
Since the inertia matrix is positive definite we have that  $m_{aa}(q_u), m_{uu}(q_u)$ are positive functions for all $q_u$. 
\end{itemize}

Now, using {\bf A1}, the EL equation \eqref{lagr} takes the form
{\small
\begin{align}
m_{uu}(q_u) \ddot q_u + c_a(q_u) \dot q_u^2 + \bar c_u(q_u) \dot{q}_a^2  + \nabla_{u} V_u(q) &\,=  0 \nonumber \\
m_{au}(q_u) \ddot q_u+m_{aa}(q_u) \ddot q_a+c_p(q_u) \dot q_a \dot q_u +c_s(q_u) \dot q_a^2 +\nabla_{a}V_u(q) &\, =  \tau \nonumber \\  \lab{newelsys1}
\end{align}
}with gravity terms $\nabla_{q_u} V_u(q)= \nabla_u V_u(q)$, $\nabla_{q_a} V_u(q)= \nabla_a V_u(q)$ and  Coriolis terms $c_a$, $\bar c_u$, $c_p$, $c_s$ $: \rea \rightarrow \rea$.  
 Thus, there exists a collocated partial linearizing control   $u_{\tt PL}:\mathbb{S} \times \rea \times \rea^2 \to \rea$ of the form \cite{Spong1994}
 \begin{align*}
 u_{\tt PL}(q,\dot q)=&\, R(q_u) u +c_p(q_u) \dot q_a \dot q_u +c_s(q_u) \dot q_a^2 +\nabla_{a}V_u(q)  \\
  & -m_{au}(q_u)m^{-1}_{uu}(q_u) [\bar c_u(q_u) \dot q_a^2 +c_a(q_u) \dot q_u^2+\nabla_{u}V_u(q) ]
 \end{align*}
 with $R(q)=[m_{aa}(q_u)-m_{au}^2(q_u)m^{-1}_{uu}(q_u)]$, such that the system \eqref{newelsys1} in closed loop with the static state--feedback control law
$
\tau=u_{\tt PL}(q,\dot q)
$
leads us to 
\begin{align}
\ddot q_a = & u \lab{newelsys}\\
m_{uu}(q_u) \ddot q_u +  \bar c_u(q_u) \dot{q}_a^2+c_a(q_u) \dot q_u^2  + \nabla_{u} V_u(q) = &-m_{au}(q_u) u.  \nonumber
\end{align}
At this point, it is important to highlight some properties of the system \eqref{newelsys}:

\begin{enumerate}
\item There does not exist a Lagrangian function verifying the EL equation \eqref{newelsys}.
\item There is not an energy function verifying the passivity of the dynamical model \eqref{newelsys}.
\item The potential energy is not separable, {\it i.e.,} there are not functions $V_x \in \rea$ and $V_y \in \rea$, so that $V_u(q_u,q_a) = V_{x}(q_a)+V_{y}(q_u)$. See the model of the Pendubot system in Section \ref{sec5}.
\end{enumerate}

Since we are addressing underactuated systems with 2 DOF's, the dynamical model \eqref{newelsys} can be written in the form of \eqref{sys}
with  ${\bfx}=\col({q_u,\,q_a,\,\dot q_u,\, \dot q_a})$ and
\begin{align}
 f(\bfx) = &
 \begmat{ \bfx_3 \\ \bfx_4 \\ - {1\over m_{uu}(\bfx_1)} \big( \bar c_u(\bfx_1)\bfx_4^2 +c_a(\bfx_1)\bfx_3^2+\nabla_{u} V_u(\bfx_1, \bfx_2)\big) \\ 0 }
  \label{fx}
\\
 \label{gx}
  g(\bfx) =& 
   \left[ \begarr{cccc} 0 & 0 & -{ m_{au}(\bfx_1) \over m_{uu}(\bfx_1)}  & \blue{1}
 \endarr \right]^\top,
\end{align}
In this paper, we are interested in the following orbital stabilization problem.

\noindent {\bf Problem Formulation}.  Find a mapping $u(\bfx)$ such that, the closed loop $\dot\bfx = f(\bfx) + g(\bfx) u(\bfx)$, with $f(\bfx)$, $g(\bfx)$ given by \eqref{fx} and \eqref{gx} respectively, has a non-trivial periodic solution $X: \rea_+ \to \rea^4$, {\em i.e.},
$$
\{{\bfx} \in \rea^4\;|\; \bfx=X(t),\;0 \leq t \leq T\},\quad T>0
$$
that is orbitally attractive, and the systems state is bounded. 
%
\vspace{-.47cm}
\section{Main result}
\label{sec4}
Before presenting the main result, the following assumptions are announced. 

 \begite

\item[{\bf A2.}] The term  $m_{uu}(\bfx_1)$ is bounded, and $m_{au}(\bfx_1)$ and $\bar c_u(\bfx_1)$ are not equal to zero in the set $\mathbb{S}$.

\item[{\bf A3.}]  There exist two smooth, bounded and non-zero functions ${\mathfrak s}(\bfx_1)$ and ${\mathcal K} (\bfx_1)$ verifying the following. 

\begite
\item[{\bf (a)}] The function ${\mathcal K} (\bfx_1)$ satisfies
\begin{equation}
\label{maths}
{\mathcal K}'(\bfx_1)= \frac{{\mathfrak s}(\bfx_1)-m_{uu}(\bfx_1)}{m_{au}(\bfx_1)}
\end{equation}
or equivalently, $\mathfrak s(\bfx_1)$ defined as 
$
{\mathfrak s}:=m_{uu}+m_{au} {\mathcal K}'. 
$

\item[{\bf (b)}] The function
\begin{equation}
\label{m}
 m(\bfx_1)= \exp \left({-2  \int_0^{\bfx_1} \beta(s)  ds}\right)
 \end{equation}
 is positive on the set $\mathbb{S}$ with 
\begin{equation}
 \hspace{-.1cm} \beta(\bfx_1)=-{{{\bar c_u (\bfx_1){\mathcal K}'(\bfx_1)^2+c_a(\bfx_1)+m_{au}(\bfx_1) {\mathcal K}''(\bfx_1)}}\over{{\mathfrak s}(\bfx_1)}}.
 \label{bet}
\end{equation}

\item[{\bf (c)}] The function $U(\bfx_1)$ has an isolated minimum point\footnote{It can be extended to the case with several isolated minima, see Remark \ref{remark:minima} for a discussion.}, around which is of interest to generate periodic oscillations, and it is given by
 \begin{equation}
 \label{pote}
 U(\bfx_1)= - \int_0^{\bfx_1} \rho(s) m(s) ds
 \end{equation}
with
\begin{equation}
\label{rho}
\rho(\bfx_1)=-{{ \nabla_uV_u(\bfx_1 , {\mathcal K}(\bfx_1))}\over{{\mathfrak s}(\bfx_1)}}.
\end{equation}
\endite
\endite

\begin{remark}
Clearly, from {\bf A2} we may impose the well-posedness of $({\mathfrak s}-m_{uu})/{m_{au}}$ on a closed interval, and the existence of $\calk$ can be guaranteed since it only depends on a scalar variable.
\end{remark}


%
\vspace{.3mm}

Now, we are in position to describe the main result of the note.
\begin{proposition}
\label{pro2}
Consider  the system \eqref{sys} with $f(\bfx)$ and $g(\bfx)$ given by \eqref{fx} and \eqref{gx} respectively, with the controller\footnote{The whole controller $\tau$ contains the pre-feedback $u_{\tt PL}$ and the term $u$. The latter in the paper resembles the widely popular VHC feedback \cite{MAGCON}.}
\begin{align}
\label{controlu}
u = &-{{\mathcal K}' (\bfx_1) \left(\bfx_4^2 \bar c_u(\bfx_1)+ \bfx_3^2  c_a(\bfx_1) \right) -m_{uu}(\bfx_1) {\mathcal K}'' (\bfx_1)\bfx_3^2\over{{\mathfrak s}(\bfx_1)}} \nonumber \\
& - 
{
{{\mathcal K}'  (\bfx_1)\nabla_u V_u(\bfx_1, \bfx_2) +{m_{uu}}(\gamma_1 \phi_1(\bfx)+\gamma_2 \phi_2(\bfx)) }\over{{\mathfrak s}(\bfx_1) }
}
,
\end{align}
 function
\begin{equation}
\label{phi}
\phi(\bfx) 
 =  \left[ \begarr{c} \bfx_2-{\mathcal K}(\bfx_1)  \\
\bfx_4-{\mathcal K}'(\bfx_1) \bfx_3
 \endarr \right]
\end{equation}
and the mapping $ {\mathfrak s}(\bfx_1)$, ${\mathcal K}(\bfx_1)$ verifying {\bf A3(a)}, in which $\gamma_1, \gamma_2>0$ are the controller gains. Then, following the  I\&I technique described in Appendix, the controller solves the problem of orbital stabilization with non-trivial orbits.
\end{proposition}
\begin{proof}
We verify the conditions {\bf Y1}-{\bf Y4} of the I\&I method  making use of the following target system
\begin{eqnarray}
\label{tsys}
 \dot{\xi}_1&=& \xi_2, \\
 \dot{\xi}_2&= & \rho(\xi_1)+ \beta(\xi_1)\xi_2^2,
\end{eqnarray}
with $\rho(\xi_1)$ and $\beta(\xi_1)$ given by \eqref{rho} and \eqref{bet}, respectively, and the mapping
\begin{equation}
\label{pi}
\pi(\xi) =   \left[ \begarr{cccc} \xi_1 & {\mathcal K}(\xi_1) &  \xi_2& {\mathcal K}' (\xi_1) \xi_2
 \endarr \right]^\top.
\end{equation}
First, let us verify {\bf Y1}, {\em i.e.}, the target dynamics \eqref{tsys} has periodic solutions. To the end, we construct a function
  \begin{equation}
 \label{Hxi}
 {\mathcal H}(\xi) := \frac{1}{2} m(\xi_1) \xi_2^2 + U(\xi_1)
 \end{equation}
with $U$ defined in \eqref{pote} and $m$ defined in \eqref{m}, then yielding
$$
\begin{aligned}
 {\partial \calh \over \partial \xi_1} 
 & ~=~ - m_1(\xi_1) [\beta(\xi_1)\xi_2^2 + \rho(\xi_1)]
 \\
 {\partial \calh \over \partial \xi_2} & ~=~ m(\xi_1)\xi_2.
\end{aligned}
$$
Hence, the target dynamics \eqref{tsys} can be written in a Hamiltonian form
\begin{equation}
\label{tsyst_H}
\dot \xi = J(\xi) \nabla \calh(\xi),\quad 
J := \begmat{0 & {1\over m(\xi_1)} \\ - {1\over m(\xi_1)} & 0},
\end{equation}
in which $\calh$ is the total energy, and $U(\xi_1)$ may be referred as the potential energy. Since the time derivative of $\calh$ along the trajectories of \eqref{tsyst_H} verifies 
$\dot \calh = 0$, thus making it a undamped (conservative) system. For a given initial condition $\xi(0)$, its trajectory is forward invariant in the set
$$
\Omega_{\xi(0)}:= \{\xi\in \rea^2~|~ \calh(\xi) = \calh(\xi(0))\}.
$$

Now, let us characterize the topology of level sets of the function $\calh(\xi)$. By solving $\nabla \calh = 0$ and invoking the assumption {\bf A3(c)} and the fact $m(\xi_1)>0$, the Hamiltonian $\calh(\xi)$ has a global (isolated) minimum point at $(\xi_1^*, 0)$ for some $\xi_1^*$. We consider the following auxiliary system 
\begin{equation}
\label{aux_syst}
\dot\eta = [ J({\eta}) - R] \nabla \calh({\eta}), \quad R := rI_2
\end{equation}
for some $r>0$ with $\eta\in \rea^2$. The time derivative of $\calh$ along the auxiliary system \eqref{aux_syst} satisfying
$$
\dot \calh = - \|\nabla \calh(\eta)\|_R^2 \le 0.
$$
From some basic Lyapunov analysis, we may obtain that $(\xi_1^*,0)$ is a globally asymptotically stable equilibrium for the system \eqref{aux_syst}, and $\calh(\eta)$ is qualified as a Lyapunov function. According to \cite[Thm. 1.2]{WIL}, for any $c$ satisfying
$
\calh(\xi_1^*,0) < c < \sup \calh(\eta),
$
 the level set $\{{\eta}\in\rea^2 ~|~\calh ({\eta})= c\}$ is diffeomorphic to $\mathbb{S}^1$. Now, let us come back to the target dynamics \eqref{tsyst_H}. Since there is no equilibrium in the level set $\{\xi\in\rea^2 ~|~\calh (\xi)= c\}$ with $c> \calh(\xi_1^*,0)$, all the trajectories of \eqref{tsyst_H}, excluding the isolated equilibrium, are non-trivial periodic solutions. Thus, we have verified the condition {\bf Y1}.

Regarding the condition {\bf Y2}, a feasible left annihilator of \eqref{gx} is
\begin{equation}
g^\perp=  \left[ \begarr{cccc} 0 & 0 &1 &m_{au}(\xi_1)\over{m_{uu}(\xi_1) }
 \endarr \right],
\end{equation}
and then, with mapping $\pi$ given by \eqref{pi}, the FBI equation in {\bf Y2} takes the form
\begin{align}
0=& \, {\bar c_u(\xi_1){\mathcal K}'(\xi_1)^2 \xi_2^2+ c_a(\xi_1) \xi_2^2+\nabla_u V_u(\xi_1, {\mathcal K}(\xi_1)) \over{ m_{uu}(\xi_1)}}  + \beta(\xi_1) \xi_2^2  \nonumber \\
& \rho(\xi_1) +{m_{au}(\xi_1)\over{m_{uu}(\xi_1)}}  \Big [ {\mathcal K}'' \xi_2^2  + {\mathcal K}' (\xi_1) \rho(\xi_1) + {\mathcal K}' (\xi_1) \beta(\xi_1) \xi_2^2\Big],  \nonumber \\
  \end{align}
which can be written as
{\small
\begin{align}
 0=&\, { c_a(\xi_1) +\bar c_u(\xi_1) {\mathcal K}'(\xi_1)^2\over{ m_{uu}(\xi_1)}} +\beta(\xi_1)+{m_{au} (\xi_1)\over{m_{uu}(\xi_1)}}  \Big[ {\mathcal K}' (\xi_1) \beta(\xi_1) + {\mathcal K}'' \Big]\nonumber \\
 0=&   -{\nabla {V_u(\xi_1, {\mathcal K}(\xi_1))} \over{ m_{uu}(\xi_1)}}  -\rho(\xi_1) -{m_{au}(\xi_1)\over{m_{uu}(\xi_1)}} {\mathcal K}' (\xi_1) \rho(\xi_1).
\end{align}
}
We note that these two equations may be rewritten as
\begin{align}
\label{beta2}
\beta(\xi_1)=&-\frac{m_{au}(\xi_1) {\mathcal K}''_1(\xi_1) +\bar c_u (\xi_1) {\mathcal K}'(\xi_1)^2 +c_a(\xi_1)}{m_{uu}(\xi_1)+m_{au}(\xi_1) {\mathcal K}'(\xi_1)} \\
\rho(\xi_1)=&-\frac{\nabla_u V_u (\xi_1, {\mathcal K(\xi_1))}}{m_{uu}(\xi_1)+m_{au}(\xi_1){\mathcal K}'(\xi_1) }.
\end{align}
%

 Moreover, both equations are non-singular if the denominator is different from zero. Hence, invoking {\bf A3(a)} we have that
\begin{equation}
m_{uu}(\xi_1)+m_{au}(\xi_1) {\mathcal K}'(\xi_1) ={\mathfrak s}(\xi_1)
\end{equation}
with free mapping ${\mathfrak s}(\bfx_1)\not = 0$, or equivalently expressed as
\begin{equation}
{{\mathcal K}'(\xi_1)}=\frac{{\mathfrak s}(\xi_1)-m_{uu}(\xi_1)}{m_{au}(\xi_1)}.
\end{equation}
On the other hand, after some straightforward calculations one finds that  the control $c(\pi(\xi))$ in \eqref{c} is given by
\begin{align}
c(\xi) &=-\frac{1} {{\mathfrak s}(\xi_1)}  \Big[ {\mathcal K}'(\xi_1) \big[\bar c_u (\xi_1) {\mathcal K}'(\xi_1)^2 +c_a(\xi_1) \big]\xi_2^2  +     \nonumber \\
& \hspace{1.5cm} +{\mathcal K}'(\xi_1) \nabla_u V_u\left(  \xi_1 ,{\mathcal K} (\xi_1)\right) -m_{uu}(\xi_1) {\mathcal K}''(\xi_1)\xi^2_2  \Big].
\label{controlc}
\end{align}

The implicit manifold description of condition {\bf Y3} is satisfied selecting $\phi(x)$ as \eqref{phi}. To complete the design we need to verify condition {\bf Y4}. Thusly, using \eqref{phi} we define the off-the-manifold coordinate
\begin{eqnarray}
\label{zs}
z_1= \bfx_2 -{\mathcal K}(\bfx_1) , \quad 
z_2=\bfx_4-{\mathcal K}' (\bfx_1) \bfx_3,
\end{eqnarray}
 taking its time derivative yields\footnote{ Given the limit of pages--and for clarity--some arguments have been omitted. }
 \begin{align}
 \dot z_1~ =&~z_2 \nonumber \\
 \dot z_2~ =& ~ u \frac{ {\mathfrak s}(x_1)}{m_{uu}} +\frac{{\mathcal K}' }{m_{uu}} \Big[   \bar c_u  \bfx_4^2 +c_a \bfx_3^2 + \nabla_u V_u(\bfx_1,\bfx_2)\Big]- {\mathcal K}'' x_3^2, \nonumber
 \end{align}
 in which we have used \eqref{maths} to get the second identity. We choose the control
$u = v(\bfx, z)$ as
 \begin{eqnarray}
 \label{controlv}
 v(\bfx, z)&=& {-{\mathcal K}'   [\bar c_u \bfx_4^2+ c_a \bfx_3^2]-  {\mathcal K}' \nabla_u V_u+m_{uu}{\mathcal K}'' \bfx_3^2  \over{{\mathfrak s}(\bfx_1)}} \nonumber \\
 &&-{ m_{uu} (\gamma_1 z_1 + \gamma_2 z_2) \over {\mathfrak s}(\bfx_1)} 
 \end{eqnarray}
 which, considering \eqref{controlc}, verifies the constraint \eqref{concon}. Furthermore, from the fact that $z=\phi(\bfx)$, the controller  \eqref{controlv} is equivalent to \eqref{controlu}.
\\
Now, the closed-loop dynamics is
\begin{align}
\dot z_1 =&z_2   \nonumber \\
\dot z_2  = & -\gamma_1 z_1 -\gamma_2 z_2 \nonumber  \\
\dot \bfx_1=& \bfx_3  \nonumber \\
\dot \bfx_2= &\bfx_4 \nonumber  \\
\dot \bfx_3=& \beta(\bfx_1)\bfx_3^2  -\frac{\nabla_u V_u (\bfx_1, {\mathcal K}(\bfx_1)+z_1)}{{\mathfrak s}(\bfx_1)} \nonumber \\
&-\frac{\bar c_u(\bfx_1) \left[ z_2 +2 {\mathcal K}'(\bfx_1) \bfx_3\right]z_2}{{\mathfrak s}(\bfx_1)} 
+ \frac{m_{au}(\bfx_1)}{{\mathfrak s}(\bfx_1)} \left( \gamma_1 z_1 + \gamma_2 z_2 \right) \nonumber  \\
\dot \bfx_4=& v(\bfx, z).
\label{fsys}
\end{align}

The first two equations ensure that $z(t) \rightarrow 0$ exponentially fast. Moreover, from {\bf A2}, $\bfx_1 \in \mathbb{S}$ and using the first equality of \eqref{zs}, the variable $\bfx_2$ remains bounded.

In the above, we have verified the exponential convergence $|z| \to 0$ as $t\to\infty$ and the boundedness of $(\bfx_1,\bfx_2)$. The remainder is to show boundedness of the other system states.

To prove the boundedness of  $\bfx_3$ we make use of {\bf A3} and the subsystem $\bfx_1$ and $\bfx_3$ of \eqref{fsys}. Hence, from {\bf A3(a)} and {\bf A3(b)} there exist constants $s_{\min}$,  $s_{\max}$,  $m_{\min}$ and $m_{\max}$ verifying 
\begin{equation}
\label{cots}
0<  s_{\min} \leq  {\mathfrak s}(s) \leq s_{\max}, \quad 
0 <m_{\min} \leq m(s) \leq m_{\max}  
\end{equation}
for all  $s \in \mathbb{S}$.

Now, computing the time derivative of the energy function ${\mathcal H}_x (\bfx_1,\bfx_3) :={\mathcal H} (\xi_1, \xi_2)|_{\xi_1=\bfx_1,\xi_2 = \bfx_3}$ given by \eqref{Hxi}, along the dynamics  $\dot \bfx_1$ and $\dot \bfx_3$ defined in \eqref{fsys} we obtain

\begin{align}
\dot {\mathcal H}_x =&\, \bfx_3 m(\bfx_1) \dot \bfx_3 - \bfx_3^3\beta(\bfx_1) m(\bfx_1) -m(\bfx_1) \rho(\bfx_1) \bfx_3  \nonumber \\
%
=&    {\nabla_uV_u (\bfx_1,{\mathcal K}(\bfx_1))- \nabla_u V_u (\bfx_1,{\mathcal K}(\bfx_1)+{z_1})  \over{\mathfrak s}(\bfx_1)}m(\bfx_1) \bfx_3  \nonumber  \\
&-\frac{\bar c_u(\bfx_1) \left[ { z_2} +2 {\mathcal K}'(\bfx_1) \bfx_3\right]{ z_2}}{{\mathfrak s}(\bfx_1)}  m(\bfx_1)  \bfx_3 \nonumber \\
&+\frac{m(\bfx_1)}{{\mathfrak s}(\bfx_1)}\bfx_3 m_{au}(\bfx_1){(\gamma_1z_1+\gamma_2z_2)}.
\label{dotH}
\end{align}
Regarding the first term, it yields
$$
\begin{aligned}
& ~\Big|\nabla_u V_u (\bfx_1,{\mathcal K}(\bfx_1))- \nabla_u V_u (\bfx_1,{\mathcal K}(\bfx_1)+z_1)\Big|
\\	
= & ~
\left|\int_0^1  {\partial^2 V_u \over \partial \bfx_2 \partial \bfx_1}
(\bfx_1, \calk(\bfx_1) + sz_1) ds \cdot z_1\right|
\\
\le & ~\int_0^1  \left|{\partial^2 V_u \over \partial \bfx_2 \partial \bfx_1}
(\bfx_1, \calk(\bfx_1) + sz_1)\right| ds \cdot |z_1|.
\end{aligned}
$$
Due to $\bfx_1 \in \mathbb{S}$, $\calk(\bfx_1) + sz_1 \in \call_\infty$ for $s\in [0,1]$ and the continuity of ${\partial^2 V_u \over \partial \bfx_2 \partial \bfx_1}$, we have that
$$
 \left|{\partial^2 V_u \over \partial \bfx_2 \partial \bfx_1}
(\bfx_1, \calk(\bfx_1) + sz_1)\right| < + \infty
$$
uniformly. Since $z$ is exponentially convergent to zero, the term
$$
\Big[\nabla_uV_u (\bfx_1,{\mathcal K}(\bfx_1))- \nabla_uV_u (\bfx_1,{\mathcal K}(\bfx_1)+z_1) \Big]
$$
in \eqref{dot_Hx} also converges to zero exponentially fast. On the other hand, the terms ${\mathfrak s}(\bfx_1), ~\bar c_u(\bfx_1),~ \mathcal{K}'(\bfx_1)$ and $m(\bfx_1)$ are all bounded due to the continuity of these functions and $\bfx_1 \in \mathbb{S}$. Hence, the time derivative of $\calh_x$ given by \eqref{dotH} takes the form
\begequ
\label{dot_Hx}
\dot \calh_x = \epsilon_1(t)\bfx^2_3  + \epsilon_2(t) \bfx_3
\endequ
with two exponentially convergent terms $\epsilon_1$ and $\epsilon_2$ defined as
$$
\begin{aligned}
 \epsilon_1 & :=
 -  2{\bar c_u(\bfx_1)\mathcal{K}'(\bfx_1)z_2 m(\bfx_1) \over \ms(\bfx_1)}
 \\
 \epsilon_2 & :=
 {\nabla_uV_u (\bfx_1,{\mathcal K}(\bfx_1))- \nabla_u V_u (\bfx_1,{\mathcal K}(\bfx_1)+z_1)  \over{\mathfrak s}(\bfx_1)}m(\bfx_1)
 \\
 &~
\quad -
 {\bar c_u(\bfx_1) z_2^2 m(\bfx_1) \over \ms(\bfx_1)}
 +
 \frac{m(\bfx_1)}{{\mathfrak s}(\bfx_1)} m_{au}(\bfx_1)(\gamma_1z_1+\gamma_2z_2).
\end{aligned}
$$
Thus, there exist $k_1,k_2,a_1$ and $a_2>0$ such that
$$
\begin{aligned}
|\epsilon_1(t)| \le a_1 e^{-k_1t}, \quad
|\epsilon_2(t)| \le a_2 e^{-k_2t}.
\end{aligned}
$$

Invoking the definition of $\calh_x$ and $\calh_x \ge 0$, as well as the equality \eqref{dot_Hx}, we have
$$
\begin{aligned}
\dot \calh_x & ~\le~ |\epsilon_1|\calh_x + |\epsilon_2|\sqrt{\calh_x}
\\
& ~\le ~ a_1 e^{-k_1t} \calh_x + a_2 e^{-k_2t}\sqrt{\calh_x}.
\end{aligned}
$$

Now, we construct the auxiliary system
\begequ
\label{dot_r}
\dot r = a_1 e^{-k_1t} r + a_2 e^{-k_2t}\sqrt{r},
\endequ
with initial condition $r(0) >0$. According the comparison lemma \cite{KHA}, if the system state $r$ in \eqref{dot_r} is essentially bounded, then the Hamiltonian function $\calh_x(x)$ would be bounded as $t\to \infty$. Solving the ordinary differential equation \eqref{dot_r}, we obtain the solution
$$
r(t) = \left(
{ \displaystyle{ \int_0^t} \hal a_2 \exp\left(\hal {a_1\over k_1} e^{-k_1s}\right) e^{-k_2s} ds + C_1
\over
\exp\left( \hal {a_1\over k_1} e^{-k_1 t} \right)
}
\right)^2
$$
with some constant $C_1$, for which the followings hold
$$
\begin{aligned}
\lim_{t\to\infty}\exp\left( \hal {a_1\over k_1} e^{-k_1 t} \right)
& = 1
\\
 \limsup_{t\to\infty}   \displaystyle{ \int_0^t} \hal a_2 \exp\left(\hal {a_1\over k_1} e^{-k_1s}\right) e^{-k_2s} ds &< +\infty.
\end{aligned}
$$
Indeed, the second limit follows by
$$
\exp\left(\hal {a_1\over k_1} e^{-k_1s}\right) \le \exp\left(\hal {a_1\over k_1}\right) : = a_3, \quad \forall s\ge 0
$$
due to $k_1>0$ and 
$$
\begin{aligned}
  \left|\displaystyle{ \int_0^\infty} \hal a_2 \exp\left(\hal {a_1\over k_1} e^{-k_1s}\right) e^{-k_2s} ds\right|
 \le & 
 \left|\displaystyle{ \int_0^\infty} \hal a_2 a_3 e^{-k_2s} ds\right|
 \\
 \le &
 a_2a_3 \int_0^\infty e^{-k_2 s} ds
 \\
 < & + \infty,
\end{aligned}
$$
due to $k_2>0$. Therefore, we have that $r \in \call_\infty$. As a consequence, $\calh_x$ and $\bfx_3 $ both are bounded. On the other hand, since the variable $z_2 = \bfx_4 - \calk'(\bfx_1)\bfx_3 $ is bounded, the state $\bfx_4$ is also bounded. We complete the proof.
\end{proof}

\begin{remark}\rm
\label{remark:minima} 
{From {\bf A3(c)}}, we assume that there is a \emph{single} isolated minimum point of the function $U$, in order to simplify the analysis. However, it can be extended to the case with multiple isolated minima, because the topological property of Lyapunov functions also holds true for local asymptotic stability in the domain of attraction, see \cite[Sec. 2]{WIL} for example. For this case with a given $c$, the level set of $\calh(\xi)$ contains several disconnected closed orbits, and we still achieve orbital stabilization with the proposed controller. The orbit, which the trajectory ultimately converges to, depends on the initial condition $\bfx(0)$.
\end{remark}
\begin{remark} \rm
Proposition \ref{pro2} does not claim that $\bfx$
converges to a particular periodic orbit, but only to one of the all possible periodic orbits generated by the target dynamics, whose existence has been assured in the proof of {\bf Y1}. We also underline that the identified conditions in this section, as well as the utilization of collocated pre-feedback (to obtain Spong's Normal Form), makes the main result not coordinate invariant. 
\end{remark}

\section{Discussions}
\label{sec:discussion}

In this section, we give some discussions about the proposed method.
\vspace{.1cm}

{\bf D1)} In this note, we address underactuated mechanical systems that after a pre-feedback do not preserve the EL structure, in contrast with  \cite{Ortetal20} and \cite{ROMetaltcst}, where the EL structure is preserved. This fact is key to prove boundedness of the whole system because the resulting model does not have gyroscopic terms depending on the actuated velocity; see \cite[Eq.(4)]{ROMetaltcst}. Besides, a similar model is studied in \cite{ROMetalejc}, which, however, is concerned with the point regulation. Clearly, the boundedness condition {\bf Y4} becomes more challenging in orbital stabilization than point regulation in \cite{ROMetalejc}. 
\vspace{.1cm}

{\bf D2)} The extension of this work to mechanical systems with $n$-DOF's verifying the form of the inertia matrix \eqref{matM}---with appropriated dimensions---and underactuation degree one is straightforward, {\em i.e.}
 \begin{equation*}
 M (q)= \left[ \begarr{cc} m_{uu} (q_u) & m_{au}^\top (q_u) \\ m_{au} (q_u) & m_{aa}  (q_u) \endarr \right],
 \end{equation*}
with $q_u \in \mathbb{S}$ and $q_a \in \rea^{n-1}$. Note that, after of prefeedback the unactuated dynamics of \eqref{newelsys} remains unchanged, thus the proposed design in Section IV holds for that class of systems with $n$-DOF's. Thus, systems like the {\it cart with the double pendulum} and {\it rotary double pendulum} with the second pendulum underactuated can be tackled \cite{BROLI}. The interested reader may refer to \cite{ROMetaltcst} for a comprehensive understanding.
\vspace{.1cm}

{\bf D3)} Clearly, the definition of mapping ${\calk}(\bfx_1)$ has the main role in the verification of the immersion condition {\bf Y2}; and ${\mathfrak s}(\bfx_1)$ is an adhered degree of freedom for that condition. The mapping $m(\bfx_1)$ is the new inertial matrix for the undamping mechanical system chosen as the target dynamics \eqref{tsys}, with $\beta(\bfx_1)$ an auxiliary function. 
\vspace{.1cm}

{\bf D4)} Let us briefly summarise the difference and similarity between the proposed scheme and the VHC approach \cite{MAGCON}.

\begin{itemize}
    
    \item[1)] These two approaches use different assumptions to guarantee the target/reduced dynamics admit periodic behaviors. In VHC, conditions for the reduced dynamics to be Lagrangian or to have periodic solutions are well-known in literature, e.g., some symmetry- and odd function-type conditions are used \cite{MOHetal}; the current work relies on the geometric structure of level surfaces of Lyapunov function \cite{WIL}.

    \item[2)] Lastly but most importantly, the boundedness of full states in closed loop has been carefully addressed in the I\&I framework. The key condition to prove this is the assumption on state space, i.e., $q \in \mathbb{S} \times \rea$. Though any trajectory of the \emph{zero} dynamics on the constraint manifold in VHC is bounded, in the transient stage (when the state converges to the constraint manifold) there is an asymptotically decaying term $\epsilon_t$ appears in the \emph{internal} dynamics, which can be viewed as an ``undamped'' Hamiltonian equation perturbed by the disturbance $\epsilon_t$. It may happen that the states are driven unbounded by $\epsilon_t$ as $t\to\infty$. In the following, we give an example to illustrate this point.
\end{itemize}

\begin{example}
Consider a two-dimensional port-Hamiltonian system
\begequ
\label{pH:example}
\dot \xi = \begmat{0 & 1 \\ -1 & 0 } \nabla \calh  + \et
\endequ
with the Hamiltonian
$$
\quad \calh(\xi) = \hal \ln(\xi_1^2 +1 ) + \hal \xi_2^2,
$$
which is \emph{conservative} when the input $\et = 0$. We draw in Fig. \ref{fig:target1} one of its solutions from $[-1, 1]^\top$, behaving as oscillation. It may be roughly seen as the ``zero dynamics'' on the constrained manifold. Now, let us consider an exponentially decaying term with $\et = [\exp(-{1\over 5}t), -2 \exp(-{1\over 5}t)]^\top$. For this case, the system \eqref{pH:example} may be viewed as the ``internal dynamics'' during the transient stage on approach to the manifold. We are able to find unbounded solution from the same initial conditions; see Fig. \ref{fig:target2}. \qed
\end{example}
\begin{figure}[htp!]
\centering
\subcaptionbox{\label{fig:target1}
Conservative Hamiltonian system with $\et=0$ (zero dynamics)}
{\includegraphics[width=0.23\textwidth]{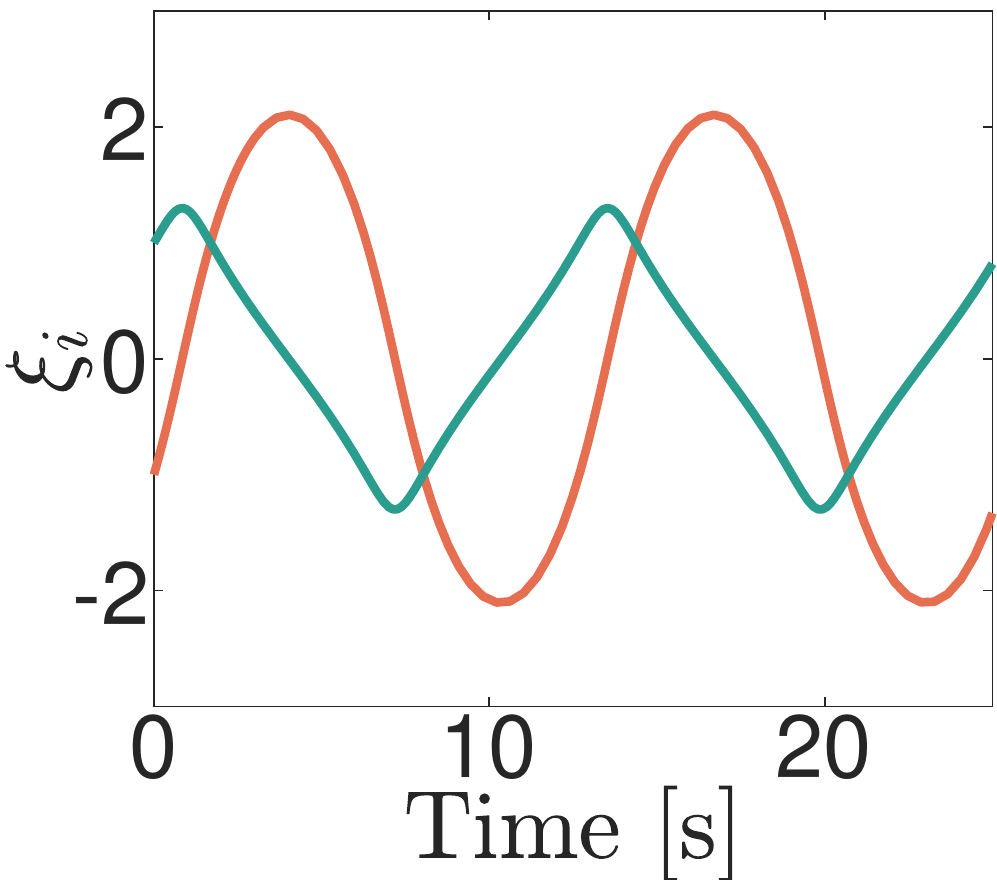}}
~
\subcaptionbox{\label{fig:target2}
Perturbed by an exponentially decaying term (internal dynamics)}
{\includegraphics[width=0.23\textwidth]{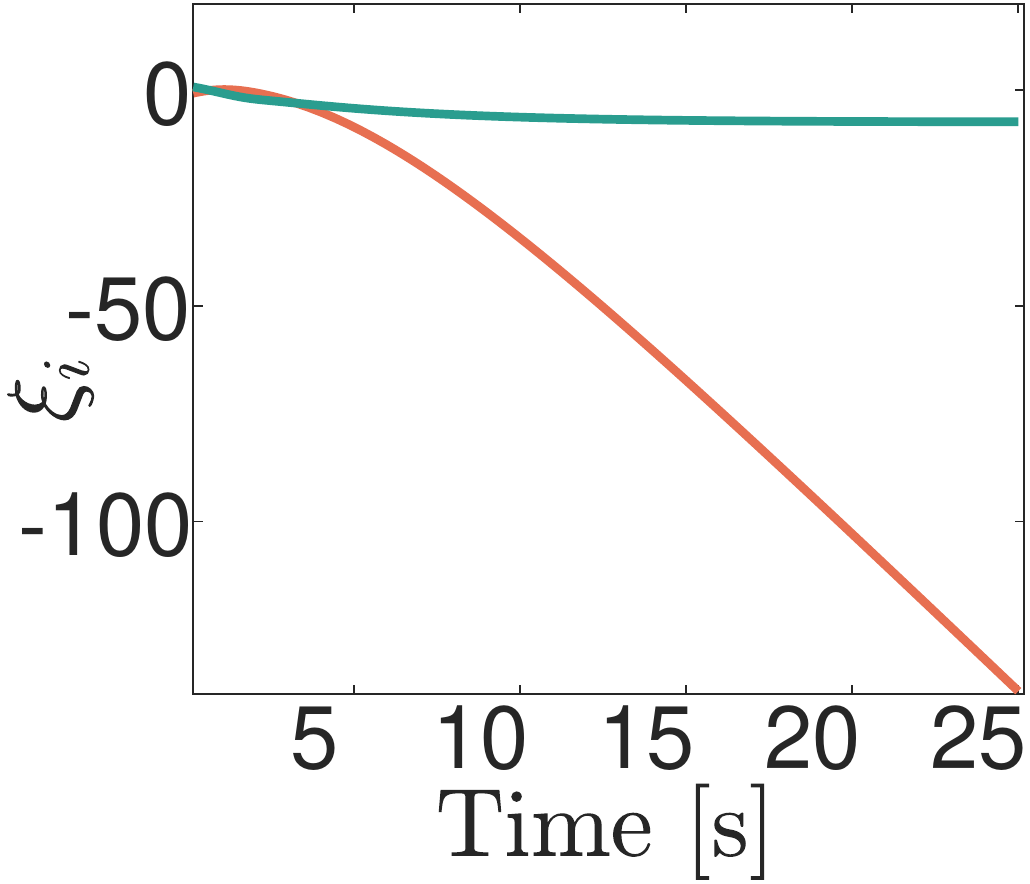}}
\end{figure}

\section{Illustrative examples}
\label{sec5}

In this section we valid our approach through simulations of two well-known models:  The Furuta pendulum and the Pendubot system, which are appreciated in Fig. \ref{fig1} on the left and the right hand side respectively. 

 \begin{figure}[htp]
 \centering
  {
    \includegraphics[width=0.2\textwidth]{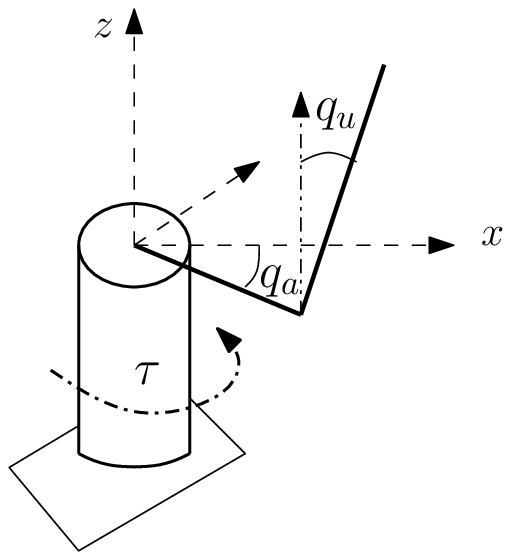}}
    \quad
 {
    \includegraphics[width=0.2\textwidth]{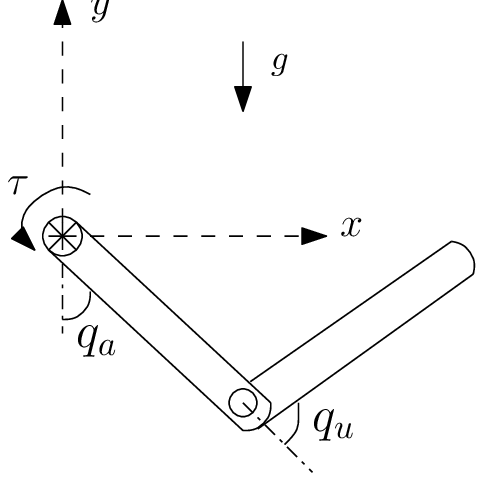}}
 \caption{Furuta's pendulum  and Pendubot system}
    \label{fig1}
\end{figure}

\subsection{Furuta pendulum}

The Furuta pendulum is a 2-DOF underactuated system, which is composed of an arm shaft (corresponding to the angle $q_a$ referred to the $x$-axis) that is subjected to an input control and a unactuated pendulum shaft (angle $q_u$ referred to the $z$-axis). See  the left hand side of Fig. \ref{fig1}. 
  The inertia matrix of the system is given by 
$$
M(q_u)= {\mathcal J} \left[ \begin{array}{cc} 
1 & a_1 \cos(q_u) \\ a_1 \cos(q_u) & a_2 + \sin^2(q_u)
\end{array}\right],
$$
with Coriolis matrix 
$$
C(q_u, \dot q)= {\mathcal J} \left[ \begin{array}{cc} 
0 & -\dot q_a \sin(q_u)\cos(q_u) \\ \begin{array}{c} -a_1 \dot q_u\sin(q_u) +\\
\dot q_a \sin(q_u) \cos(q_u) \end{array}&  \dot q_u \sin(q_u) \cos{q_u}
\end{array}\right],
$$
potential energy $V(q_u)= a_3 \cos(q_u)$, with constants $a_1=\frac{mrl}{{\mathcal J}}$, $a_2= \frac{{\mathcal J}_a +mr^2}{{\mathcal J}}$, $a_3=mgl$. Here,  $m$ is the mass of the pendulum, $r$ is the length of the arm, $2l$ is the total length of the pendulum, ${\mathcal  J}$ is the moment of inertia of the pendulum and ${\mathcal J}_a$ is the moment of inertia of the arm and the motor. See \cite{Acosta10} for further details of the model and parameters, in particular Section 2.1 therein. 

Now, following Proposition \ref{pro2} we select 
$
{\mathfrak s} =-k_1 
$
with  $k_1$ a free gain. Thus, 
  \begin{eqnarray*}
  {\mathcal K}(\bfx_1)&=& -(1+ k_1) \frac{\ln \left( \sec(\bfx_1) +\tan(\bfx_1)  \right) }{a_1} \\
  {\mathcal K}' (\bfx_1)&=& -\frac{1+k_1}{a_1 \cos(\bfx_1)}, \quad  
  {\mathcal K}'' (\bfx_1)= -\frac{(1+k_1) \sin (\bfx_1)}{a_1 \cos^2 (\bfx_1)}. 
  \end{eqnarray*}

Consequently, the functions $\rho(x_1)$ and $\beta(x_1)$ given by \eqref{rho} and \eqref{bet} respectively, take the form
\begin{eqnarray*}
\rho(\bfx_1)&=& -\frac{a_3}{{\mathcal J}} \sin(\bfx_1), \quad  
\beta(\bfx_1)= - \kappa_1 \tan(\bfx_1)
\end{eqnarray*}
with $\kappa_1=\frac{(1+k_1)(1+k_1+ a_1^2)}{ a_1^2 k_1}$. Now, using these functions and after some straightforward calculations, the functions \eqref{pote} and \eqref{m} have the form  
\begin{eqnarray}
\label{Ud}
U(\bfx_1)&=& \frac{a_3 }{{\mathcal J} \kappa_2} \left(  \frac{1}{\cos(\bfx_1)^{\kappa_2}}  - 1 \right) \\
m(\bfx_1) &=&  \cos(\bfx_1)^{-\kappa_1}
\end{eqnarray}
with $\kappa_2=\frac{2+4k_1+ 2a_1^2 +2k_1^2+ k_1 a_1^2} {a_1^2}$. 
  
Since the desired objective is to oscillate the pendulum in the upper half plane, we impose  $k_1>0$,  so that \eqref{Ud}  is positive for $\bfx_1 \in (-\frac{\pi}{2}, \frac{\pi}{2})$ and has a minimum at zero with $\bfx_1=0$, that corresponds to the upright equilibrium point of the pendulum. 

Now, we present several simulations to validate the performance of the controller proposed. For this, we use the parameters of \cite{Acosta10} which are $m=0.0679 $Kg, $l=0.14$m, $r=0.235$ and ${\mathcal J}=0.0012 $Kgm$^2$.

 The simulations were carried out to assess the impact on the transient performance of the proposed controller using the free gain  $k_1$ and   the tuning gains $\gamma_1$ and $\gamma_2$ of the $z$-dynamics. We use  ${\bfx}(0)=[\frac{\pi}{9} (\text{rad}), 0.6 (\text{rad}), 0, 0]$ as initial condition for these scenarios. In addition, we show the effects produced by different initial conditions. 
 
 In Figs. \ref{fig:gamsx1X3} and \ref{fig:gamsx2X4} {we appreciate  that the oscillation amplitude in both links depends of the selection of  $\gamma_1$ and $\gamma_2$ (setting $k_1=5$). In particular, we have that  the amplitudes and velocities decreasing when  $\gamma_1$ increases and $\gamma_2$ decreases}. In all cases the oscillation of the unactuated link is around of the upright position. On the other hand,  from Fig. \ref{fig:k1sx1X3} we notice that preserving  $\gamma_{1,2}=5$  and increasing $k_1$,  the unactuated link  oscillates with the same amplitude--whose value is given by the initial condition---while its velocity decreases. Regarding the actuated link, both---oscillation amplitude and velocities--are modified. These behaviors are shown in Fig \ref{fig:k1sx2X4}. Finally, for different initial conditions, $\gamma_{1}, \gamma_2$ as above and setting $k_1=9$, the  amplitude of oscillation of the unactuated link is also determined by the initial condition, while for the actuated one is not. It has sense that the velocities increase as initial positions increase. This is because  a greater torque must be applied to generate the oscillations of the unactuated link on the upright position; see Figs. \ref{fig:condx1X3} and \ref{fig:condx2x4}.\footnote{An animation of the system behaviors may be found at \href{https://youtu.be/LlPHtccgcYQ}{\tt \blue{youtu.be/LlPHtccgcYQ}}.}

\begin{figure*}[htp]
\centering
\subcaptionbox{\label{fig:gamsx1X3}
Trajectories of $(\bfx_1,\bfx_3)$ with
different gains $\gamma_1$ and $\gamma_2$ and the same initial conditions}
{\includegraphics[width=0.32\textwidth, height = 4cm]{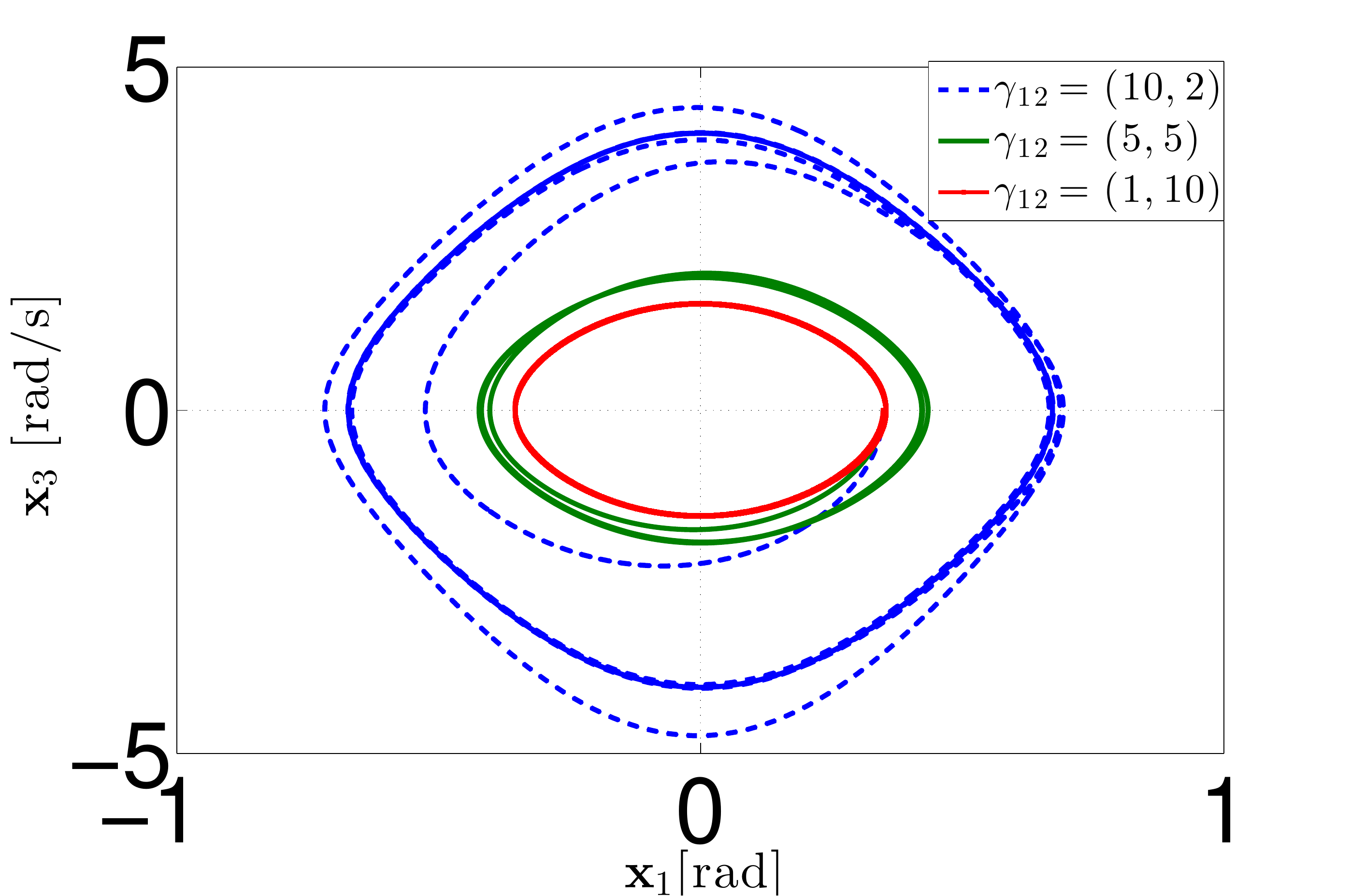}}
~
\subcaptionbox{\label{fig:gamsx2X4}
Trajectories of $(\bfx_2,\bfx_4)$ with
different gains $\gamma_1$ and $\gamma_2$ and the same initial conditions }{\includegraphics[width=0.32\textwidth, height = 4cm]{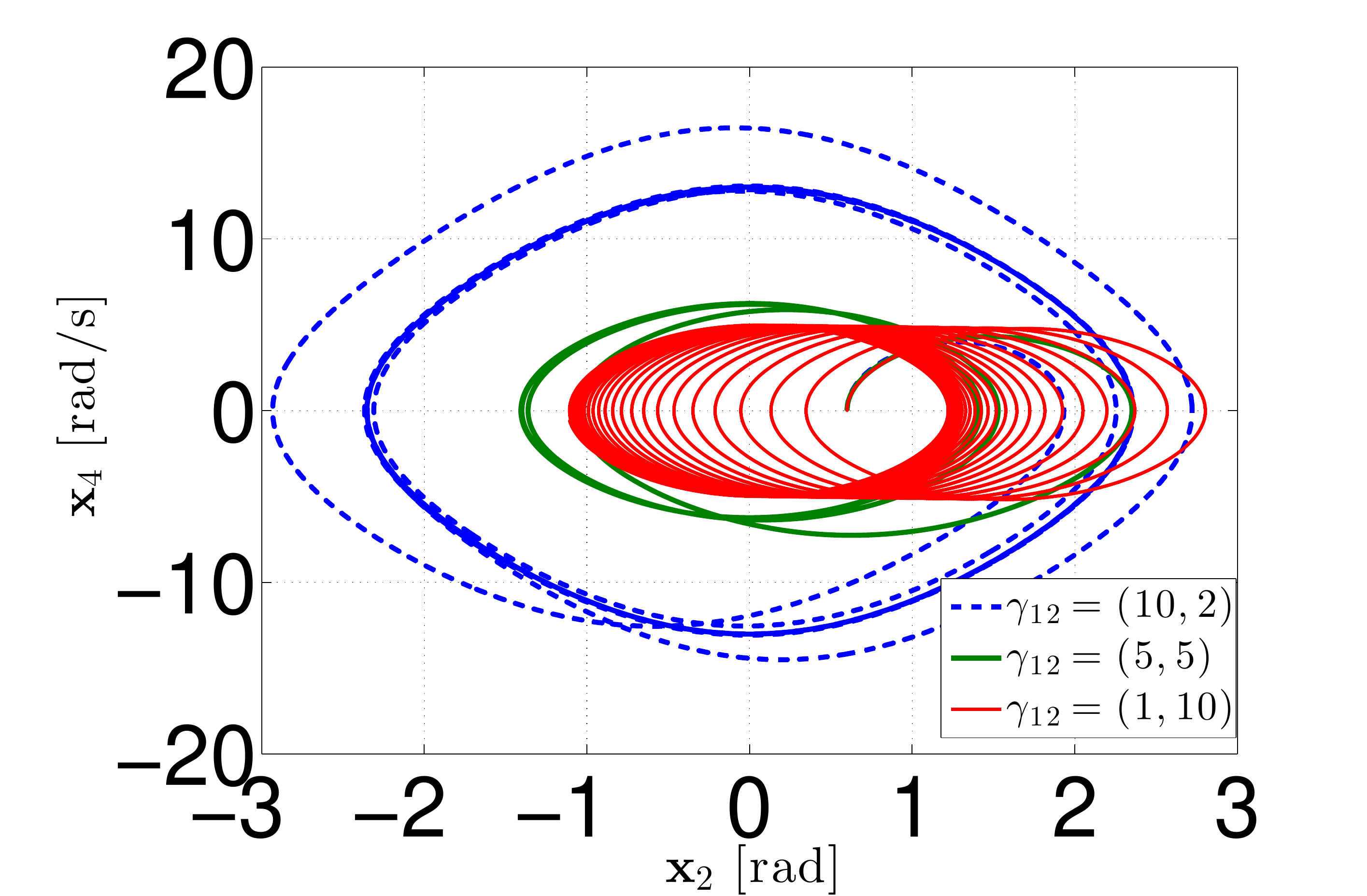}}
~
\subcaptionbox{\label{fig:k1sx1X3}
Trajectories of $(\bfx_1,\bfx_3)$ with
different gains $\gamma_1$ and $\gamma_2$ and the same initial conditions
 }{\includegraphics[width=0.32\textwidth, height = 4cm]{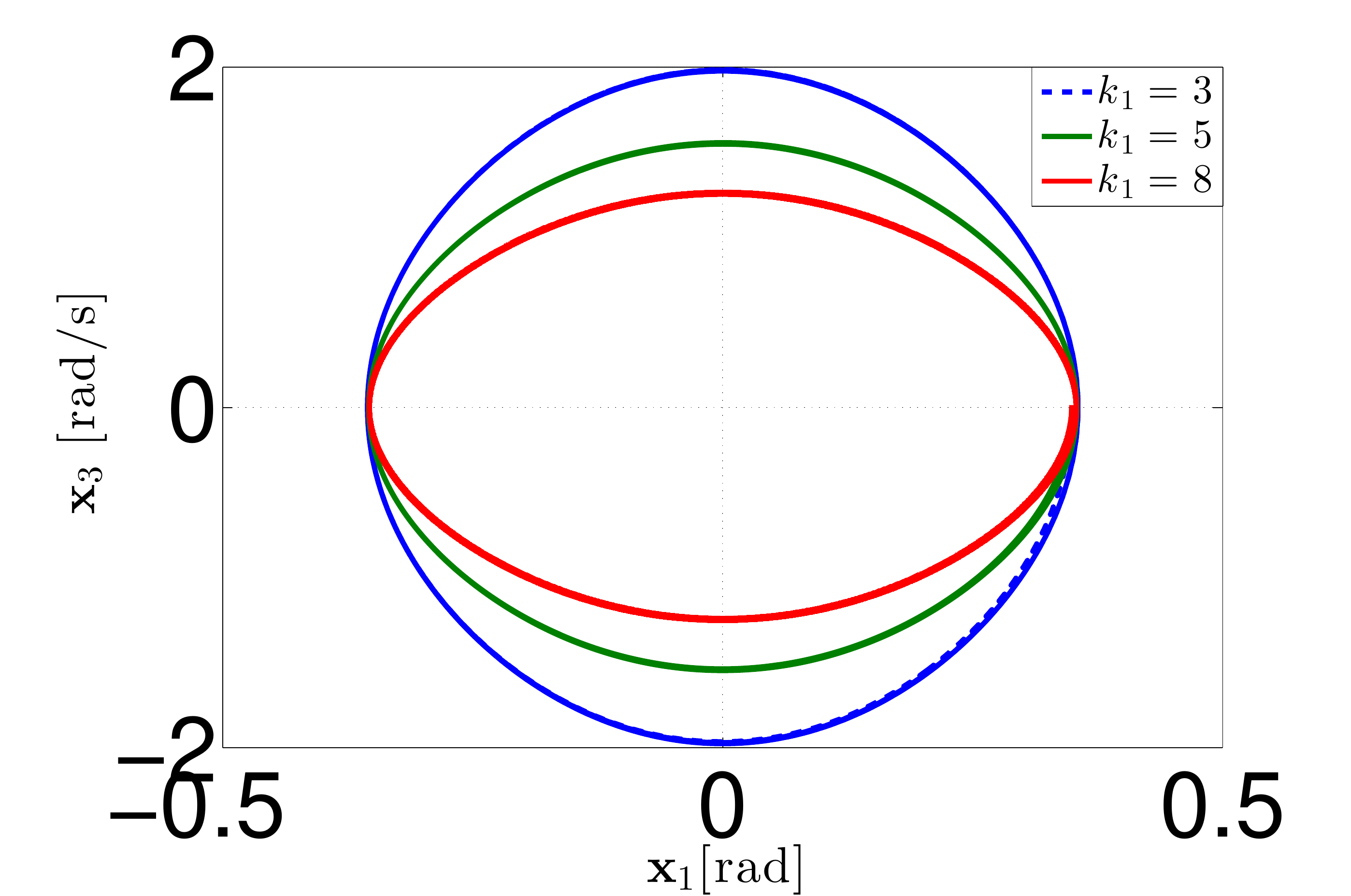}}
 \\
\subcaptionbox{\label{fig:k1sx2X4}
Trajectories of $(\bfx_2,\bfx_4)$ with
different gain $k_1$ and the same initial conditions
 }{\includegraphics[width=0.32\textwidth, height = 4cm]{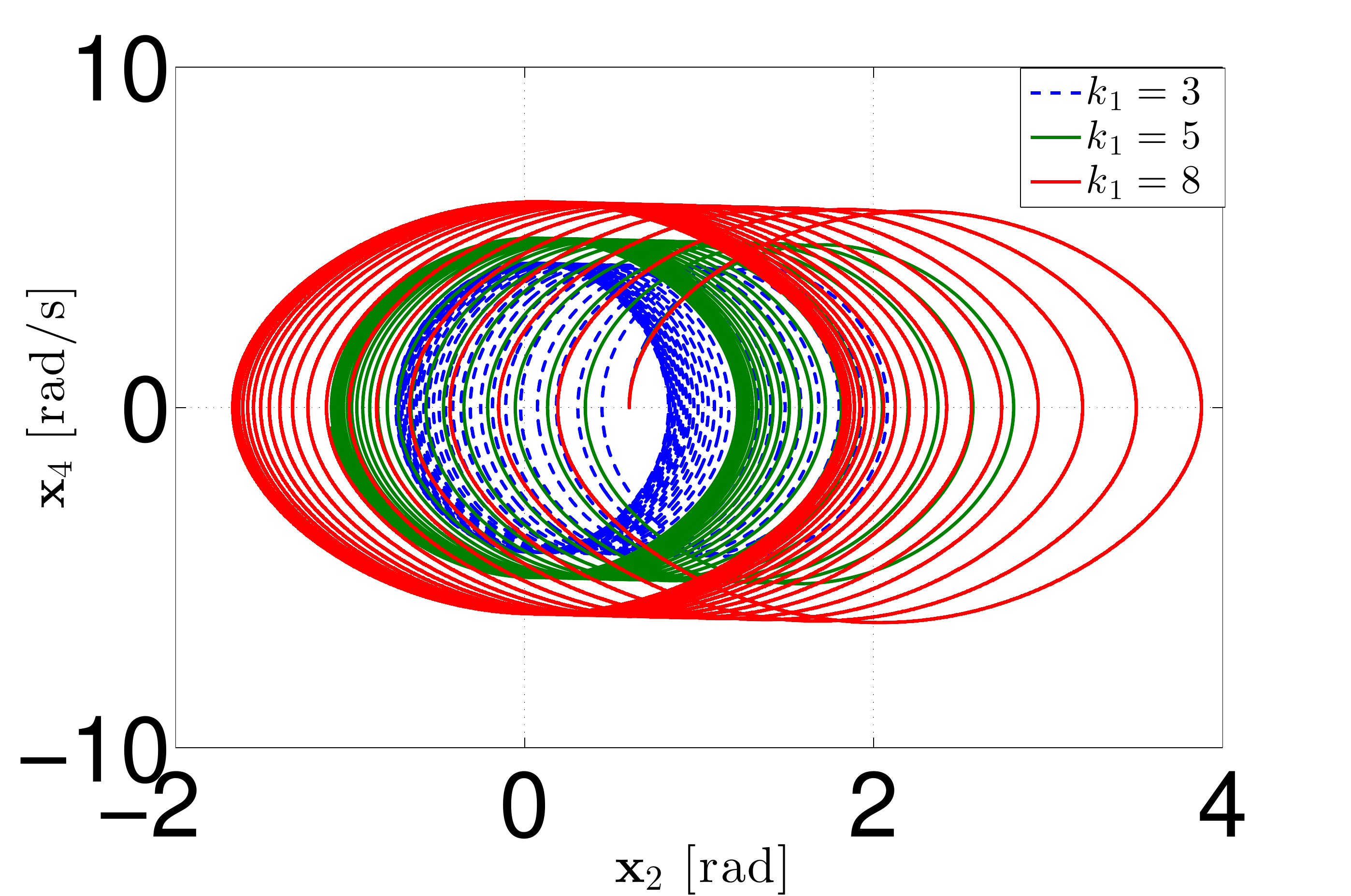}}
~
\subcaptionbox{\label{fig:condx1X3}
Trajectories of $(\bfx_1, \bfx_3)$ with
different initial conditions
 }{\includegraphics[width=0.32\textwidth, height = 4cm]{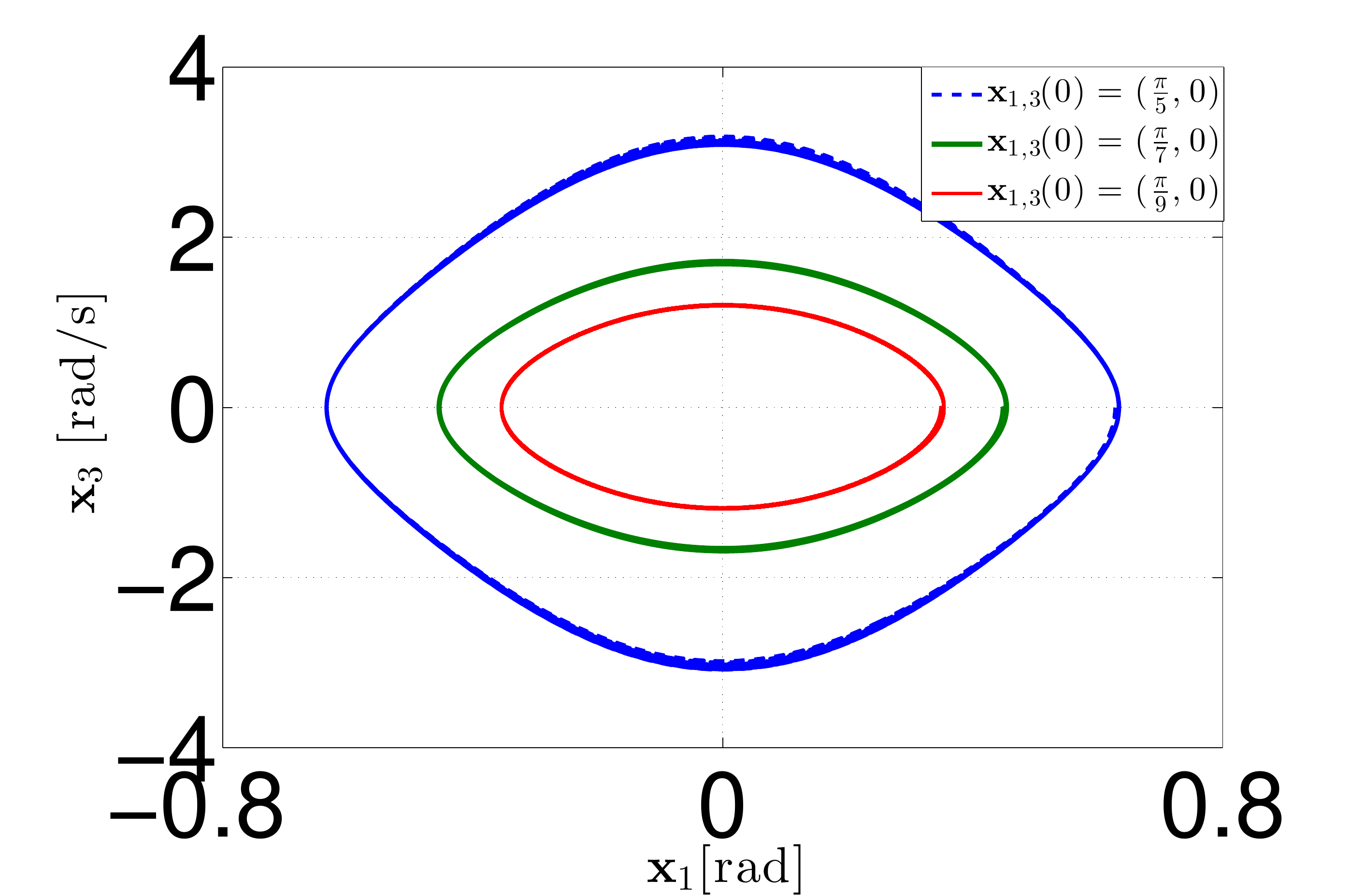}}
 ~
\subcaptionbox{\label{fig:condx2x4}
Trajectories of $(\bfx_2, \bfx_4)$ with
different initial conditions }{\includegraphics[width=0.32\textwidth, height = 4cm]{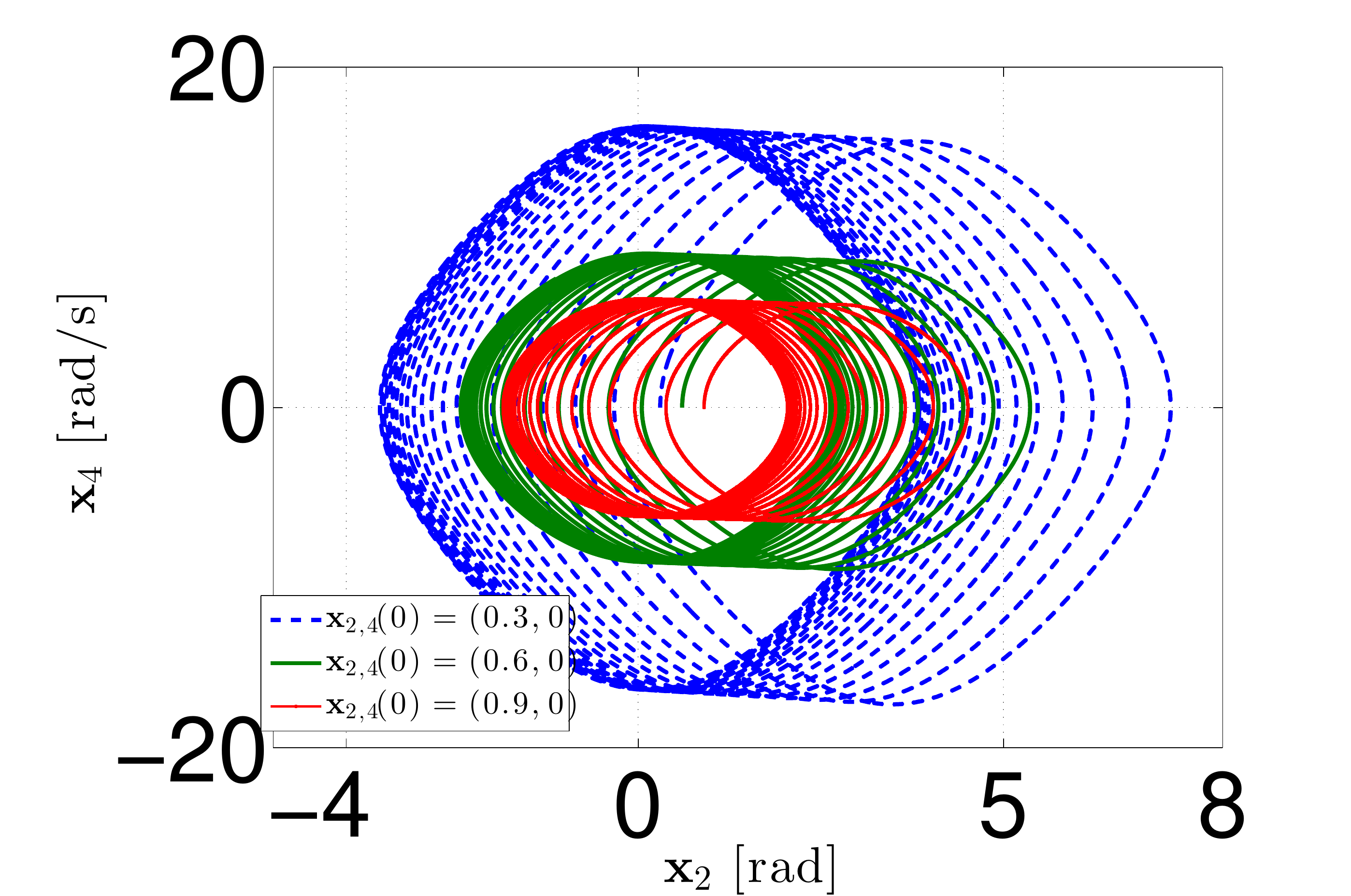}}
\caption{Simulation results for the model of Furuta pendulum}
\end{figure*}

%
%
%
%
%
%
%
%
%
%

\subsection{Pendubot system}

The Pendubot is an underactuated mechanical system consisting of two links which can move freely on a vertical plane through a pair of revolute joints. The first one has an actuator that applies a torque on it, while the  second joint is passive. A schematic picture of the Pendubot is shown on the right hand side of Fig. \ref{fig1}, where $q_a$ and  $q_u$ are the actuated and unactuated coordinates, respectively.

The inertia matrix $M(q_u)$, Coriolis matrix $C(q_u,q_u)$; and potential energy $V(q)$ are described as follows
\begin{eqnarray*}
M(q_u)&=& \left[ \begin{array}{cc} 
 c_2 & c_2+c_3 \cos(q_u) \\  c_2+c_3 \cos(q_u) & c_1+c_2+2c_3 \cos(q_u) 
\end{array}\right] \\
C(q, \dot q)&=& c_3 \sin(q_u) \left[ \begin{array}{cc} 
0 &  \dot q_a \\
 -\dot{q}_u -\dot{q}_a  &- \dot{q}_u 
\end{array}\right]
\end{eqnarray*} 
and 
$$
V(q)=-c_4 g \cos(q_a) -c_5 g \cos(q_a +q_u).
$$
The parameters $c_i$ with $i=1,\ldots,5$ are defined as $c_1=m_1 l_{c1}^2 +m_2 l_1^2 +I_1$, $c_2=m_2 l_{c2}^2 +I_2$, $c_3=m_2 l_1 l_{c2}$, $c_4= m_1 l_{c1} +m_2 l_1 $ and $c_5= m_2 l_{c2}$,  where $l_j$ is the length of the links, $l_{cj}$ is the length at the center of mass and $m_j$   is the mass of the links with $j=1,2$. For more details of the model, see \cite{Toan17} \footnote{Compared with the classical model representation of the Pendubot ---where the actuated link is the first one---we have rewritten the Inertia and Coriolis matrices to match them with the notation of the paper; see \eqref{matM}. }.

For this system, the control objective is to generate stable oscillations of the underactuated link  starting above the horizontal. Hence, following Proposition \ref{pro2},  we choose  
${\mathfrak s} (\bfx_1)=-\kappa_2 [c_2 +c_3 \cos(\bfx_1)]+c_2$
with $k_2$ a free parameter to be selected. As consequence,  we have that  ${\mathcal K}=-k_2 \bfx_1$, ${\mathcal K}'=-k_2$ and ${\mathcal K}''=0$. Using these mappings, the functions $\rho(x_1)$ and $\beta(x_1)$ are given by 
\begin{eqnarray}
\label{rhopen}
\rho(\bfx_1)&=& -\frac{c_5 g \sin( (1-k_2) \bfx_1)}{k_2 [c_2 +c_3 \cos(\bfx_1)]-c_2} \\
\beta(\bfx_1)&=& - \frac{c_3 k_2^2 \sin(\bfx_1)}{k_2 [c_2 +c_3 \cos(\bfx_1)]-c_2}.
\label{betapen}
\end{eqnarray}
We note that the motion of the unactuated link is relative to the actuated  one, thus, to ensure that the oscillations are generated on the upper plane, we select $k_2$ so that the new potential energy has a minimum around of 
\begin{equation}
\label{Udpen}
q_u=\pm n \pi, \quad  n>1, \, \,n \in \nea,
\end{equation}
which forces to the actuated link to be in an upright position. To this end, using \eqref{rhopen},  \eqref{betapen}, setting $k_2=-1$ and making simple calculations, the new potential energy \eqref{pote}  and function \eqref{m} take  the form
\begin{eqnarray}
U(\bfx_1)&=& \frac{2c_5 g (c_2 +c_3 \cos(\bfx_1))}{ c_3^2\, (4 c_2^2 +4 c_2 c_3 \cos(\bfx_1) + c_3^2 \cos(\bfx_1)^2 )} \label{Upend} \\
m(\bfx_1)&=&\frac{1}{(2c_2+c_3 \cos(\bfx_1))^2}.
\label{mpend}
\end{eqnarray}

To corroborate the effectiveness of our approach, simulation results are presented using  the parameters given in \cite{Toan17}, which are   $l_1=0.2 $m, $l_2=0. 28 $m, $m_1=0.2 $Kg, $m_2=0.052 $Kg, $l_{c1}=0.13 $m, $l_{c2}= 0.15 $m, $I_1=3.38\times 10^{-1}$Kgm$^2$ and $I_2=1.17\times 10^{-3} $Kgm$^2$.

As above mentioned, the second link is unactuated and its motion depends of the first one. Consequently, the oscillations of both links are ``synchronized''. Since our orbital controller with \eqref{Upend} and \eqref {mpend} ensures oscillations on the upper plane, Figs. \ref{pendu:gamsx1x3} and \ref{pendu:gamsx2x4} show the performance of both links under different $\gamma_{1,2}$ values and using as initial conditions $\bfx (0)= [\frac{\pi}{3}, \frac{\pi}{1.5}, 0, 0]$, thus we appreciate that the oscillating frequency  is increasing when $\gamma_1$ increases and $\gamma_2$ decreases. Now,  we set $\gamma_{1,2}=(10,5)$ and propose different initial conditions verifying  $\bfx_1(0)+\bfx_2(0) \geq \pi$. This condition allows us to have the unactuated pendulum  close to the upright position independently, that the actuated link is  located on the lower or upper plane. The transitory behavior of both links are appreciated in Figs. \ref{pendu:condx1x3} and \ref{pendu:condx2x4} and, as expected, the oscillations are generated on the upper plane around of $\bfx_1=\pi$.  This contrasts radically with \cite{FREetal}, where the oscillations  for 1) both links in the down position [the lower plane] and 2)  the upright position  for the actuated coordinate  and down position for unactuated one,  are distinguished  satisfying some {\it numerical} quantities in each case. An animation of the system behaviors is also available at the same link as the one in the first example. An animation of the system behaviors is also available at the same link as the one in the first example.

\begin{figure*}[htp]
\centering
\subcaptionbox{\label{pendu:gamsx1x3}
Trajectories of $(\bfx_1,\bfx_3)$ with
different gains $\gamma_1,\gamma_2$ and the same initial conditions}
{\includegraphics[width=0.245\textwidth]{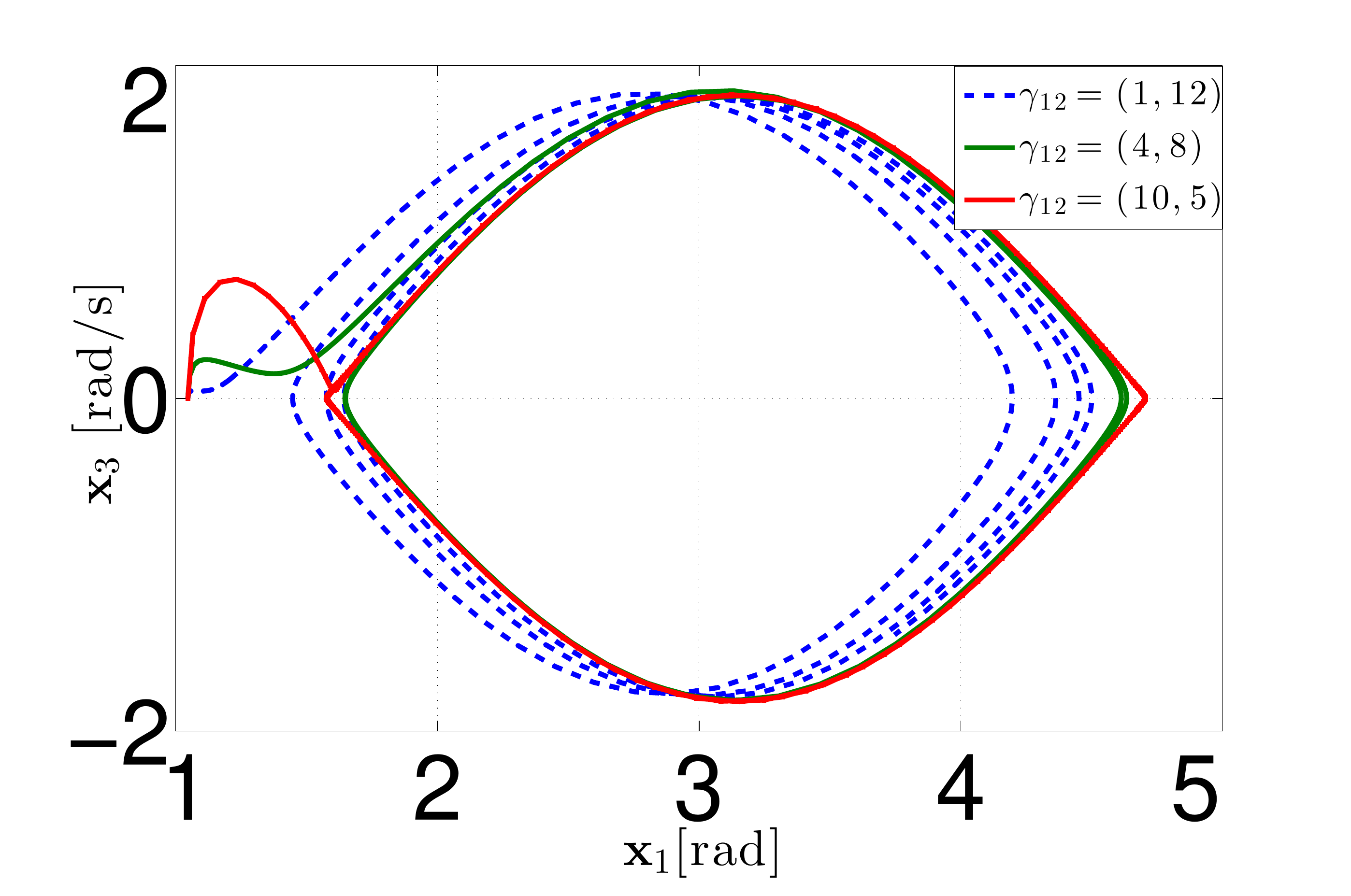}}
\subcaptionbox{\label{pendu:gamsx2x4}Trajectories of $(\bfx_2 ,\bfx_4)$ with
different gains $\gamma_1,\gamma_2$ and the same initial conditions
}
{\includegraphics[width=0.245\textwidth]{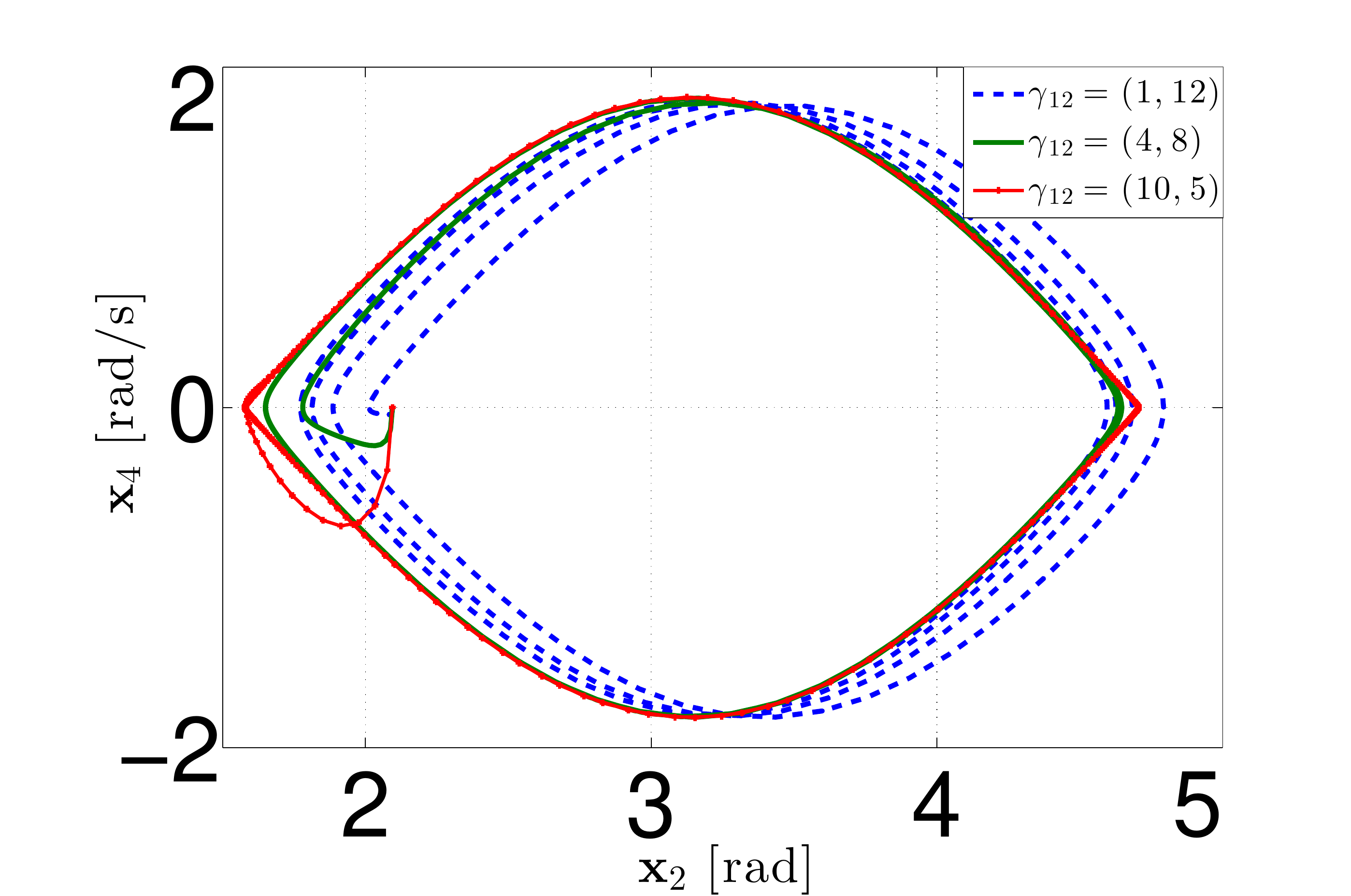}}
\subcaptionbox{\label{pendu:condx1x3}Trajectories of $(\bfx_1,\bfx_3)$ with different initial conditions
}
{\includegraphics[width=0.245\textwidth]{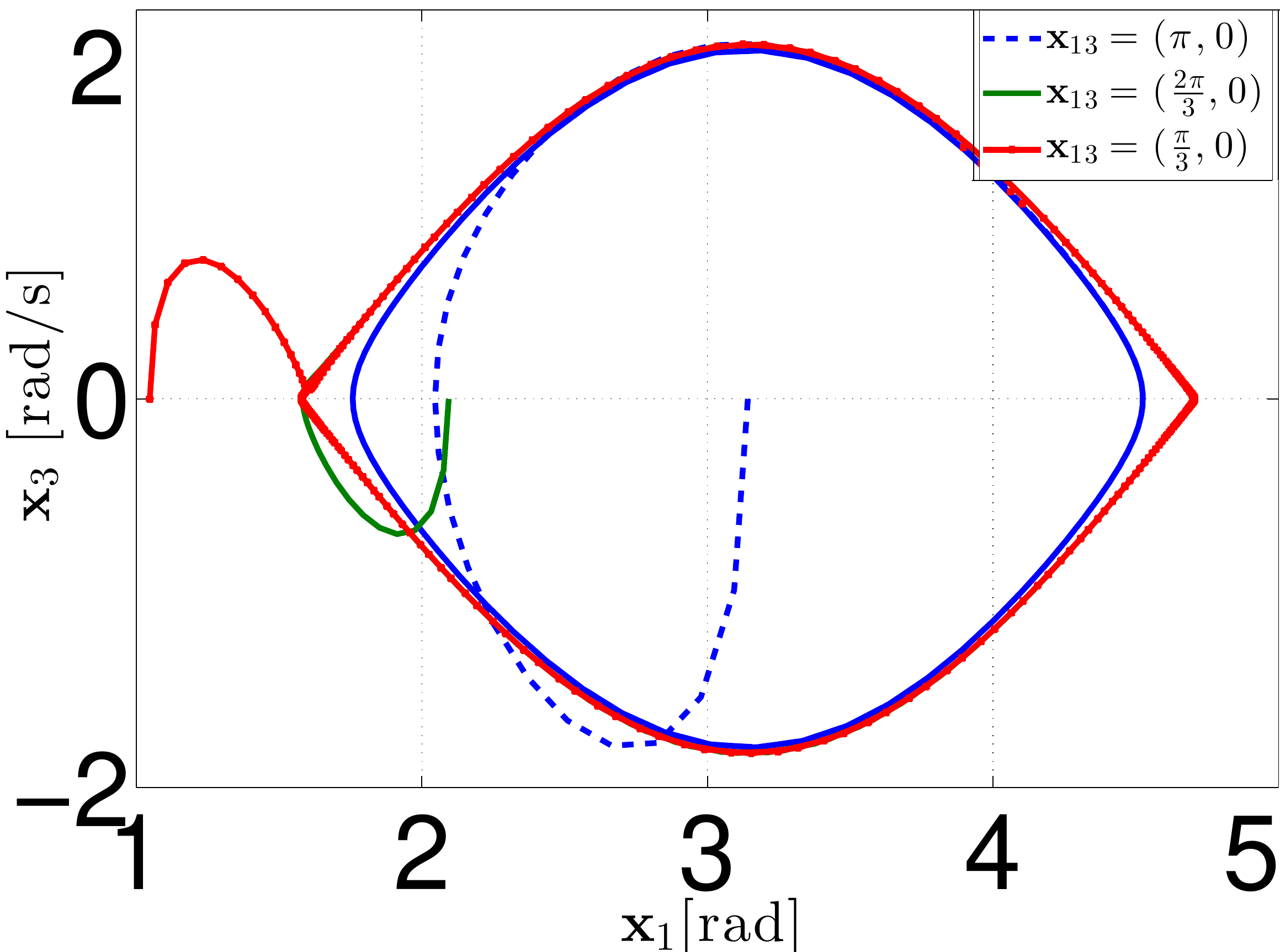}}
\subcaptionbox{\label{pendu:condx2x4}Trajectories of $(\bfx_2,\bfx_4)$ with
different  initial conditions
}
{\includegraphics[width=0.245\textwidth]{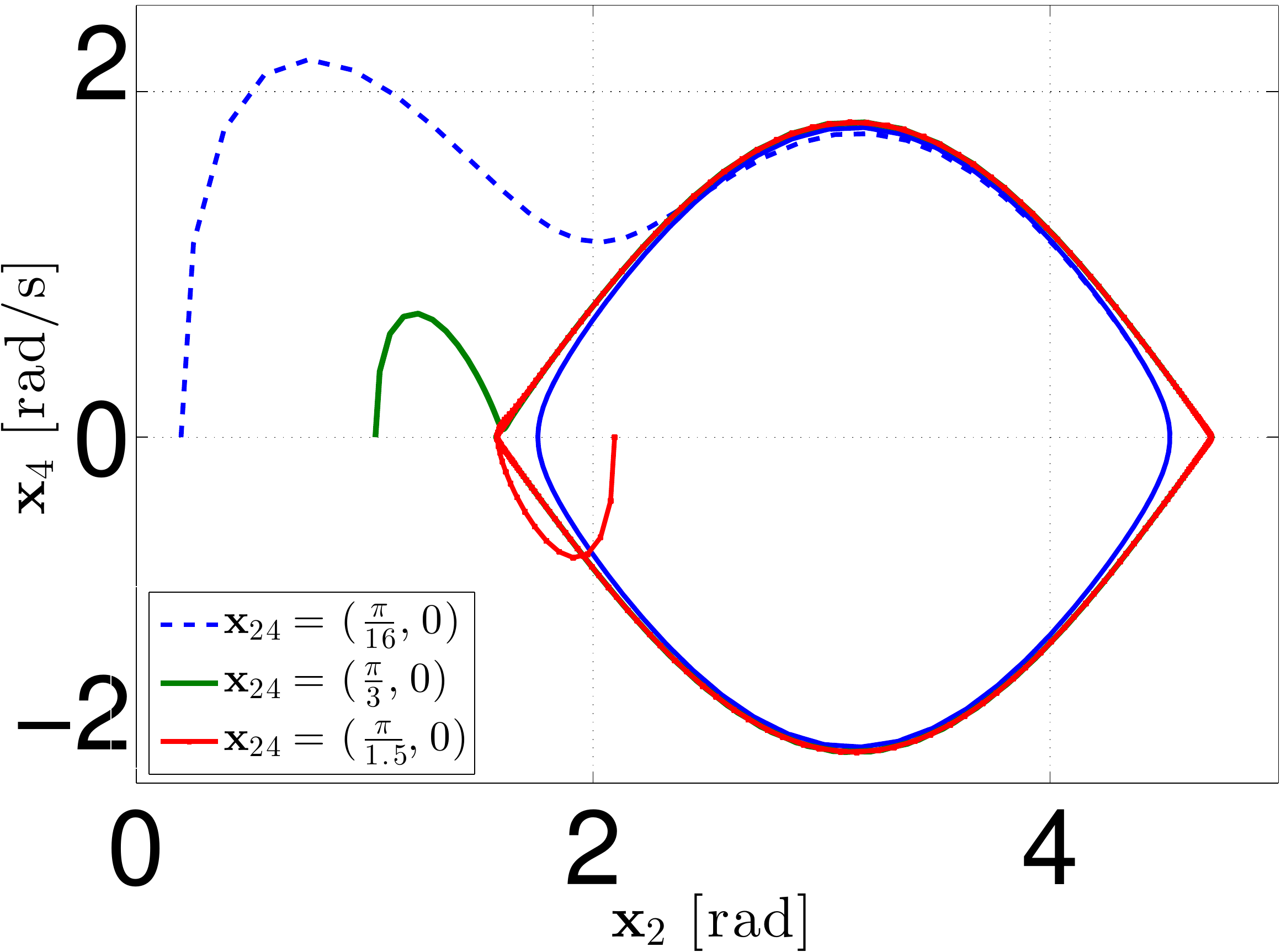}}
\caption{Simulation results for the model of Pendubot}
\end{figure*}


%
%

%
%

\section{Conclusions and Future work}
\label{sec6}

We tailored in this paper the I\&I orbital stabilization approach, which was recently proposed in \cite{Ortetal20}, to a class of underactuated mechanical models. First, we use the widely popular partial feedback linearization \cite{Spong1994} to simplify the sequential design; then, it is combined with an elaborated I\&I controller. The full design is simple and in a compact form, and all the assumptions can be easily verified as a priori rather than being done numerically on-line. In the proof of the main result, we had an encounter with a conservative EL system perturbed by a decaying term, the boundedness of which has been studied carefully and comprehensively. At the end, we apply the approach to two benchmarks, {\em i.e.}, the models of Furuta pendulum and Pendubot, showing good performance, as well as verifying all the theoretical results. Though the approach is given to the 2-DOF systems, it can be seamlessly extended to the arbitrary-DOF case with underactuation one. On the other hand, the extensions to the systems with more than one passive configurations, and to the hybrid models, are interesting topics to be explored in future, which should be practically useful in application to robotics. Moreover, as future work we will consider to avoid the use of the feedback linearization step, since it involves the exact cancellation of some nonlinear terms, which is intrinsically non-robust.

\section*{Acknowledgements}
The authors would like to express their gratitude to four reviewers for their careful reading of our manuscript, and many thoughtful and constructive comments that helped improve its clarity.

\begin{appendix}

\section*{Immersion and Invariance Method }

The immersion and invariance (I$\&$I) technique is a constructive methodology to design nonlinear and adaptive control, and state observers for dynamical systems \cite{IIbook}. In the following proposition we summarize the I\&I methodology for orbital stabilization \cite{Ortetal20}, in which the scheme is proposed for general nonlinear systems.

\begin{proposition}
\label{pro1}\rm\cite{Ortetal20}
Consider the system
\begequ
\label{sys}
\dot{\bfx} = f(\bfx) + g(\bfx)u,
\endequ
with state $\bfx\in\rea^n$, input $u\in\rea^m$ ($m < n$), and $g(\bfx)$ full rank. Assume we can find mappings
\begin{eqnarray}
&\alpha: \rea^p \to \rea^p, \quad \pi: \rea^p \to \rea^n, \quad \phi: \rea^n \to \rea^{n-p}, \\
& \quad v :\rea^n\times\rea^{n-p} \to \rea^m
\end{eqnarray}
with $p < n$, such that, the following conditions hold.
\begin{itemize}
    \item[\bf{Y1.}](Target systems) The dynamical system
    \begequ
    \label{tardyn}
     \dot{\xi} = \alpha(\xi)
    \endequ
with state $\xi \in \rea^p$, has non-trivial periodic solutions $\xi_\star (t)$ which are determined by the initial condition $\xi(0)$ in the set of interests.

    \item[\bf{Y2.}] (Immersion condition) For all $\xi\in\rea^p$, the FBI equation holds
    \begequ
    \label{fbi}
    \begin{aligned}
     g^\perp(\pi(\xi))  \varpi (\xi) & =0
     \\
     \varpi(\xi) & :=  f(\pi(\xi)) - \nabla \pi^\top(\xi) \alpha(\xi).
     \end{aligned}
     \endequ
    %
    \item[\bf{Y3.}] (Implicit manifold) The following set identity holds
    {\small
    \begequ
    \label{impman}
     \calm:=\{\bfx\in\rea^n | \phi(\bfx)=0\} =
     \{ \bfx \in\rea^n  | \bfx=\pi(\xi),\;\xi \in \rea^p\}.
    \endequ
    }
    \item[\bf{Y4.}] (Attractivity and boundedness) All the trajectories of the system
    \begequ
    \begin{aligned}
     \dot{z} & ~= ~\nabla \phi^\top(\bfx) [f(\bfx)+g(\bfx)v(\bfx,z)] \\
     \dot{\bfx} & ~=~ f(\bfx) + g(\bfx)v(\bfx,z)
    \end{aligned}
    \lab{auxsys}
    \endequ
    with the initial condition $z(0) =  \phi(\bfx(0))$ and the constraint
\begequ
\lab{concon}
v(\pi(\xi),0) = c(\pi(\xi)),
\endequ
where
\begequ
\lab{c}
c(\pi(\xi)) := -[g^\top(\pi(\xi))g(\pi(\xi))]^{-1}g^\top(\pi(\xi)) \varpi(\xi),
\endequ
are bounded and satisfy
    \begequ
    \label{ztozer}
     \lim_{t\to \infty} z(t) =0.
    \endequ
\end{itemize}
Then the closed-loop system
$
\dot{\bfx} = f(\bfx) + g(\bfx)v(\bfx,\phi(\bfx))
$
is bounded, and has a continuum of non-trivial, non-isolated periodic \emph{attractive} solutions on $\calm$.
\end{proposition}

\end{appendix}

%


\end{document}